\documentclass[12pt,twoside]{amsart}
\usepackage{amssymb}
\usepackage{amscd}

\title[semi log canonical surfaces]
{Abundance theorem for 
semi log canonical surfaces in positive characteristic} 
\author{Hiromu Tanaka} 
\address{Department of Mathematics, Imperial College London, 
180 Queen's Gate, London SW7 2AZ, UK} 
\email{h.tanaka@imperial.ac.uk}
\newcommand{\Ker}[0]{{\operatorname{Ker}}}

\newcommand{\Spec}[0]{{\operatorname{Spec}}}
\newcommand{\Aut}[0]{{\operatorname{Aut}}}

\newcommand{\Supp}[0]{{\operatorname{Supp}}}

\newcommand{\Ex}[0]{{\operatorname{Ex}}}
\newcommand{\Sing}[0]{{\operatorname{Sing}}}
\newtheorem{thm}{Theorem}[section]
\newtheorem{lem}[thm]{Lemma}
\newtheorem{cor}[thm]{Corollary}
\newtheorem{prop}[thm]{Proposition}

\newtheorem{ex}[thm]{Example}

\theoremstyle{definition}

\newtheorem{dfn}[thm]{Definition}

\newtheorem{rem}[thm]{Remark}
\newtheorem*{ack}{Acknowledgments}      
\newtheorem{nota}[thm]{Notation}         
\newtheorem{step}{Step}
\newtheorem{nasi}[thm]{}

\newcommand{\MO}{\mathcal{O}}

\newcommand{\Q}{$\mathbb{Q}$}
\newcommand{\Z}{$\mathbb{Z}$}
\begin{document}

\maketitle

\begin{abstract}
We prove the abundance theorem 
for semi log canonical surfaces in positive characteristic. 
\end{abstract}

\tableofcontents

\setcounter{section}{-1}

\section{Introduction}
A semi log canonical (for short, slc) $n$-fold 
is a generalization of log canonical (for short, lc) $n$-folds. 
In this paper, we prove 
the abundance theorem for slc surfaces 
in positive characteristic. 
We use the same definition of slc varieties as the one of \cite{Kollar}. 

\begin{thm}
Let $(X, \Delta)$ be a projective slc surface 
over an algebraically closed field of positive characteristic. 
If $K_X+\Delta$ is nef, 
then $K_X+\Delta$ is semi-ample. 
\end{thm}

Let us briefly review the history of the semi log canonical varieties in characteristic zero. 
The notion of semi log canonical singularities 
is introduced in \cite{KSB} for a moduli problem. 
The abundance theorem for slc surfaces 
is proved in \cite{AFKM} and \cite{KeKo}. 
\cite{Fujino} generalizes this result to dimension three. 
Moreover, \cite{Fujino} shows that 
the abundance theorem for slc $n$-folds follows 
from the two parts: 
\begin{enumerate}
\item{The abundance theorem for lc $n$-folds. } 
\item{The finiteness theorem 
of the pluri-canonical representation 
for $(n-1)$-folds.}
\end{enumerate}
\cite{FG} shows that (2) holds for each $n\in\mathbb Z_{>0}$. 
If $n=3$, then (1) follows from \cite{KeMaMc}. 
If $n\geq 4$, then (1) is an open problem. 
For a recent development of the theory of slc varieties in characteristic zero, 
see \cite{Fujino2}, \cite{FG}, \cite{Gongyo} and \cite{HX}. 
For related topics, see \cite{Kawamata}, \cite{AFM} and \cite{KoMaMc}. 

In this paper, we use the strategy of \cite{Fujino}. 
Hence, we must prove (1) and (2) 
in the case where $n=2$ and ${\rm char}\,k>0$. 
In this case, (1) is a known result by \cite{Fujita}. 
It is not defficult to prove (2). 
However, \cite{Fujino} uses many fundamental results 
based on the minimal model theory and the Kawamata--Viehweg vanishing theorem. 
We can freely use the minimal model theory for surfaces in positive characteristic 
by \cite{Tanaka1} (cf. \cite{Fujita}, \cite{KoKo}). 
Although there exist counter-examples to the Kodaira vanishing theorem in positive characteristic (\cite{Raynaud}), 
we can use some weaker vanishing theorems obtained in \cite{Tanaka2} (cf. \cite{KoKo}, \cite{Xie}). 

In characteristic two, some new phenomena happen. 
For example, in characteristic zero, 
the Whitney umbrella $\{x^2=yz^2\}\subset \mathbb A^3$ is a typical example of slc surfaces 
(cf. \cite[Definition~12.2.1]{AFKM}). 
In characteristic two, this is slc but not normal crossing in codimension one. 
Moreover, \cite{Fujino} uses the following fact: 
if a field extension $L/K$ 
satisfies $[L:K]=2$ and its characteristic is zero, 
then $L/K$ is a Galois extension. 
But, in characteristic two, 
this field extension $L/K$ may be purely inseparable. 
Thus, some proofs are more complicated.

\begin{nasi}[Overview of contents]
In Section~1, we summarize the notations. 
The normalization of an slc surface is an lc surface. 
Therefore, we should investigate lc surfaces. 
Every lc surface is birational to a dlt surface. 
Thus, in Section~2, we consider 
a dlt surface $(X, \Delta)$. 
More precisely, 
we consider $\llcorner\Delta\lrcorner$ 
because $\llcorner\Delta\lrcorner$ has 
the patching data of the normalization. 
In Section~3, we calculate the normalization of nodal singularities. 
In Section~4, we prove the main theorem. 
In Section~5, we summarize fundamental results on dlt surfaces. 
These results may be well-known but 
the author can not find a good reference. 
\end{nasi}

\begin{ack}
The author would like to 
thank Professor Osamu Fujino 
for many comments and discussions. 
In particular, Professor Osamu Fujino teaches him 
Proposition~\ref{4.4} and Theorem~\ref{finitegroup} in this paper. 
He thanks Professor Atsushi Moriwaki for warm encouragement. 
He would like to thank the referee for valuable suggestions. 
\end{ack}

\section{Notations}

We will not distinguish the notations 
invertible sheaves and Cartier divisors. 
For example, 
we will write $L+M$ for 
invertible sheaves $L$ and $M$. 

Throughout this paper except for Section~3, 
we work over an algebraically closed field $k$ 
of positive characteristic and 
let ${\rm char}\,k=:p.$

In this paper, a {\em variety} means 
a pure dimensional 
reduced scheme which is separated 
and of finite type over $k$. 
A {\em curve} or a {\em surface} means 
a variety whose dimension is one or two, 
respectively. 
Note that varieties, curves and surfaces 
may be reducible. 

Let $X$ be a noetherian reduced scheme and 
let $X=\bigcup X_i$ be the irreducible decomposition. 
Let $Y_i\to X_i$ be the normalization of $X_i$. 
Then we define the {\em normalization} of $X$ by 
$\coprod Y_i\to \coprod X_i\to X.$ 
We say $X$ is {\em normal} if the normalization morphism is an isomorphism.

Let $X$ be a variety. 
We say $\Delta$ is a {\em \Q-divisor} on $X$ if 
$\Delta$ is a finite sum 
$\Delta=\sum \delta_i\Delta_i$ where 
$\delta_i\in\mathbb Q$ and 
$\Delta_i$ is an irreducible and 
reduced closed subscheme of codimension one 
which is not contained in the singular locus $\Sing(X)$. 
Note that, in this case, the local ring $\mathcal O_{X, \Delta_i}$ 
is a discrete valuation ring. 

We will freely use the notation 
and terminology in \cite{Kollar}. 
In the definition in \cite[Definition~2.8]{Kollar}, 
for a pair $(X, \Delta)$, $\Delta$ is not necessarily effective. 
But, in this paper, we assume 
$\Delta$ is an effective \Q-divisor. 
For a reducible normal variety $X$ and 
an effective \Q-divisor $\Delta$, 
we say $(X, \Delta)$ is lc (resp. dlt, klt) 
if each irreducible component is 
lc (resp. dlt, klt). 

For the definition of (nodes and) slc varieties, 
see Definition~\ref{def-node} and Definition~\ref{def-slc}.
These definitions are the same as \cite[1.41, 5.10]{Kollar}.

\section{Boundaries of dlt surfaces}

In this section, 
we investigate dlt surfaces. 
First, we consider the case of curves. 
The main result of this section is Proposition~\ref{2.1}. 
Proposition~\ref{2.1} is the surface version of Proposition~\ref{2.1curve}. 

\begin{prop}\label{2.1curve}
Let $(X, \Delta)$ be an irreducible lc curve. 
Let $f:X\to R$ be a projective surjective 
morphism such that $f_*\mathcal O_X=\mathcal O_R.$ 
Assume that $S:=\llcorner\Delta\lrcorner\neq 0$ and 
let $T:=f(S)$. 
If $K_X+\Delta\equiv_f 0$, then 
one of the following assertions holds. 
\begin{enumerate}
\item{$f_*\mathcal O_S=\mathcal O_T$.}
\item
{$f_*\mathcal O_S\neq\mathcal O_T$. 
$X\simeq \mathbb P^1$ and $\dim R=0.$ 
Moreover, $\Delta=S$ and $S$ is two distinct points.}
\end{enumerate}
\end{prop}

\begin{proof}
If $\dim R=1$, then we see $X\simeq R$ and 
we obtain (1). 
We may assume $\dim R=0.$ 
Since $\deg(K_X+\Delta)=0$ and 
$\llcorner\Delta\lrcorner\neq 0$, 
we see $X\simeq \mathbb P^1$ and 
$S$ has at most two points. 
If $S$ is one point, then we obtain (1). 
\end{proof}

In the above proposition, 
(1) is a good case. 
Hence, we classify the other case (2) as above. 
For this, we want sufficient conditions 
for $f_*\mathcal O_S=\mathcal O_T$. 




%
\begin{prop}\label{1.3}
Let $f:X\to Y$ be a projective surjective morphism 
between irreducible normal varieties such that 
$f_*\mathcal O_X=\mathcal O_Y.$ 
Assume the following conditions. 
\begin{enumerate}
\item{$(X, \Delta)$ is a \Q-factorial lc surface 
such that $(X, \{\Delta\})$ is klt.}
\item{$S:=\llcorner\Delta\lrcorner\neq 0$ and 
let $T:=f(S).$}
\item{$-(K_X+\Delta)$ is $f$-nef and $f$-big.}
\end{enumerate}
Then, 
$f_*\mathcal O_S
=\mathcal O_T.$ 
In particular, for every $y\in Y$, 
$S\cap f^{-1}(y)$ 
is connected or an empty set. 
\end{prop}

\begin{proof}
\begin{step}
In this step, we assume $\dim Y\geq 1$ and 
we prove the assertion. 
Consider the exact sequence: 
$$0\to \mathcal O_X(-\llcorner\Delta\lrcorner)\to
\mathcal O_X\to \mathcal O_{\llcorner\Delta\lrcorner}\to 0.$$
Take the push-forward by $f$: 
$$0\to f_*\mathcal O_X(-\llcorner\Delta\lrcorner)\to
\mathcal O_Y\to f_*\mathcal O_{\llcorner\Delta\lrcorner}\to 
R^1f_*\mathcal O_X(-\llcorner\Delta\lrcorner).$$
It is sufficient to prove that 
the last term 
$R^1f_*\mathcal O_X(-\llcorner\Delta\lrcorner)$ vanishes. 
Since 
$$-\llcorner\Delta\lrcorner=K_X+\{\Delta\}-(K_X+\Delta),$$
we have $R^1f_*\mathcal O_X(-\llcorner\Delta\lrcorner)=0$ by \cite[Theorem~2.12]{Tanaka2}. 
\end{step}

\begin{step}
In this step, we assume $\dim Y=0$ and 
we prove the assertion. 
It is sufficient to prove that $S$ is connected. 
Suppose the contrary and let us derive a contradiction. 
We may assume that $X$ is smooth 
by replacing $(X, \Delta)$ with $(Z, \Delta_Z)$ in Proposition~\ref{dltblowup}. 
Since $X$ is a ruled surface, we have a surjective morphism 
$$h:X \to B$$
onto a smooth projective curve $B$ or $X \simeq \mathbb P^2$. 
The latter case is clear, hence we can assume the existence of $h$. 

Since $-(K_X+\Delta)$ is $h$-nef and $h$-big, we can apply Step~1 to $h$. 
If $\llcorner \Delta \lrcorner$ contains an $h$-horizontal prime divisor, 
then we see that $\llcorner \Delta \lrcorner$ is connected by Step~1. 
Thus we may assume that every irreducible component of 
$\llcorner \Delta \lrcorner$ is $h$-vertical. 
Again, by Step~1, we can assume that $h(\llcorner \Delta\lrcorner)=\{Q_1, \cdots, Q_r\}$ with $r \geq 2$ 
and the intersection $\Supp(\llcorner \Delta\lrcorner) \cap h^{-1}(Q_i)$ is connected for every $i$.

We can run a $(K_X+\{\Delta\})$-MMP over $B$ 
that does not change the number of the connected components of $\llcorner \Delta \lrcorner$. 
Indeed, if $\Supp h^{-1}(Q_i)  \subset \Supp \llcorner \Delta \lrcorner,$ 
then we just contract a curve $C$ in $h^{-1}(Q_i)$ with $(K_X+\{\Delta\})\cdot C<0$. 
If $\Supp h^{-1}(Q_i) \not\subset \Supp \llcorner \Delta \lrcorner$, then 
we contract 
a curve $C \subset \Supp h^{-1}(Q_i)$ 
such that $C \not\subset \Supp \llcorner \Delta \lrcorner$ and $(K_X+\{\Delta\})\cdot C<0$. 
Repeating this, we end with a Mori fiber space: 
$$h:X\overset{\varphi}\to X'\overset{h'}\to B$$ 
where $\varphi:X\to X'$ is a composition of extremal birational contractions 
and $h':X'\to B$ is a Mori fiber space. 
Note that $(X', \Delta':=\varphi_*\Delta)$ is \Q-factorial lc, $(X', \{\Delta'\})$ is klt, 
and $-(K_{X'}+\Delta')$ is nef and big. 
Moreover, $\llcorner \Delta'\lrcorner$ contains at least two fibers of $h'$. 

Assume that $-(K_{X'}+\Delta')$ is ample. 
Since $\rho(X') = 2$ by \cite[Theorem~3.27(4)(b)]{Tanaka1}, 
$X'$ has the two $(K_{X'}+\Delta')$-negative extremal rays. 
The extremal ray, not corresponding to $h'$, 
induces a morphism $\psi:X' \to X''$ 
that is a birational contraction or another Mori fiber space structure to a curve. 
For each case, we can apply Step~1 to $\psi$ and obtain a contradiction.

Therefore we may assume that 
$-(K_{X'}+\Delta')$ is not ample. 
Then, by Nakai--Moishezon criterion, there exists a curve $C'$ on $X'$ 
such that $(K_{X'}+\Delta')\cdot C'=0.$ 
This implies $C'^2<0$ and $h'(C')=B.$ 
Since $\rho(X')=2$, we obtain $-(K_{X'}+\Delta')\cdot \Gamma>0$ for a curve $\Gamma \neq C'$. 
If $-(K_{X'}+\Delta')$ is semi-ample, then it induces the contraction of $C'$: $\psi:X' \to X''$ 
and 
$$\psi^{-1}\psi(C') \cap \Supp \llcorner \Delta\lrcorner=C' \cap \Supp \llcorner \Delta\lrcorner$$ 
has at least two points. 
Applying Step~1 to $\psi$, we obtain a contradiction. 
Thus it suffices to show that $-(K_{X'}+\Delta')$ is semi-ample. 
We have 
$$0=(K_{X'}+\Delta')\cdot C'> (K_{X'}+C')\cdot C'$$
because $\Delta'$ contains an $h'$-horizontal curve. 
Then $C'\simeq \mathbb P^1$ by \cite[Theorem~3.19(1)]{Tanaka1}. 
Thus, by Keel's theorem (\cite[Theorem~0.2]{Keel}), $-(K_{X'}+\Delta')$ is semi-ample. 
We are done. 
\end{step}
\end{proof}

\begin{lem}\label{2.4}
Let 
$$f:X\overset{q}\to X'\overset{f'}\to R$$ 
be projective morphisms between normal varieties such that 
$q$ is birational and $f'_*\mathcal O_{X'}=\mathcal O_R.$ 
Assume the following conditions. 
\begin{enumerate}
\item{$(X, \Delta)$ is a \Q-factorial lc surface such that $(X, \{\Delta\})$ is klt.}
\item{$\Ex(q)=:E$ is an irreducible curve. }
\item{$-(K_X+\Delta)$ is $q$-nef. }
\item{$\llcorner\Delta\lrcorner$ is $q$-nef. }
\end{enumerate}
Then, for every $r\in R$, 
the number of connected components of 
$\llcorner\Delta\lrcorner \cap f^{-1}(r)$ 
is equal to the number of connected components of 
$\llcorner q_*\Delta\lrcorner \cap f'^{-1}(r)$. 
\end{lem}

\begin{proof}
Let $q(E)=:x'_0$ and $f'(x'_0)=:r_0.$ 
If $E\cap \Supp\llcorner\Delta\lrcorner=\emptyset,$ then 
the assertion is clear. 
Thus, we may assume $E\cap \Supp\llcorner\Delta\lrcorner\neq\emptyset.$ 

We claim $q(\Supp\llcorner\Delta\lrcorner)=\Supp\llcorner q_*\Delta\lrcorner.$ 
The inclusion $q(\Supp\llcorner\Delta\lrcorner)\supset \Supp\llcorner q_*\Delta\lrcorner$ 
is clear. 
Then, it is enough to show $q(E)\in \Supp\llcorner q_*\Delta\lrcorner.$ 
If $E\not\subset \Supp\llcorner\Delta\lrcorner$, then 
$E\cap \Supp\llcorner\Delta\lrcorner\neq\emptyset$ implies 
$q(E)\in \Supp\llcorner q_*\Delta\lrcorner.$ 
On the other hand, 
if $E\subset \Supp\llcorner\Delta\lrcorner$, 
then the $q$-nefness implies that 
there exists a prime component $C\neq E$ of $\llcorner \Delta\lrcorner$ with 
$C\cap E\neq 0.$ 
We see 
$$q(E)\in q(C)\subset \Supp\llcorner q_*\Delta\lrcorner.$$
In each case, we obtain the claim. 

For every $r\in R$, we obtain 
\begin{eqnarray*}
q(\Supp\llcorner\Delta\lrcorner \cap f^{-1}(r))
&=&q(\Supp\llcorner\Delta\lrcorner \cap q^{-1}(f'^{-1}(r)))\\
&=&q(\Supp\llcorner\Delta\lrcorner) \cap f'^{-1}(r)\\
&=&\Supp\llcorner q_*\Delta\lrcorner \cap f'^{-1}(r).
\end{eqnarray*}
Assume that 
the numbers of connected components are different. 
Then there exist at least two connected components 
$X_1$ and $X_2$ of 
$\Supp\llcorner\Delta\lrcorner \cap f^{-1}(r_0)$ such that 
$x'_0\in q(X_1)$ and $x'_0\in q(X_2).$
We take the intersection 
$$\Supp\llcorner\Delta\lrcorner \cap f^{-1}(r_0)
=X_1\amalg X_2\amalg\cdots$$
with $q^{-1}(x'_0)$ and 
we obtain the following equation 
$$\Supp\llcorner\Delta\lrcorner \cap q^{-1}(x'_0)
=(X_1\cap q^{-1}(x'_0))\amalg (X_2\cap q^{-1}(x'_0))\amalg\cdots.
$$
Thus, in order to derive a contradiction, 
it is sufficient to prove that 
$\Supp\llcorner\Delta\lrcorner\cap q^{-1}(x'_0)$ 
is connected. 
Since $-(K_X+\Delta)$ is $q$-nef and $q$-big, 
we can apply Proposition~\ref{1.3}. 
Thus $\Supp\llcorner\Delta\lrcorner \cap q^{-1}(x_0')$ 
is connected. 
\end{proof}

\begin{prop}\label{4.4}
Let $f:X\to Y$ be a projective surjective morphism 
between irreducible normal varieties such that 
$f_*\mathcal O_X=\mathcal O_Y.$ 
Assume the following conditions. 
\begin{enumerate}
\item{$(X, \Delta)$ is a \Q-factorial lc surface 
such that $(X, \{\Delta\})$ is klt.}
\item{$S:=\llcorner\Delta\lrcorner\neq 0$ and 
let $T:=f(S).$}
\item{$K_X+\Delta\equiv_f 0.$}
\item{$T=f(S)\subsetneq Y$.}
\end{enumerate}
Then, 
$f_*\mathcal O_S
=\mathcal O_T.$ 
In particular, for every $y\in Y$, 
$S\cap f^{-1}(y)$ 
is connected or an empty set. 
\end{prop}

\begin{proof}
By (4), we have $\dim Y\neq 0.$ 
If $\dim Y=2,$ then 
the assertion follows from Proposition~\ref{1.3}. 
Thus we may assume  $\dim Y=1.$ 
It is sufficient to prove that 
$\mathcal O_Y=f_*\mathcal O_X\to f_*\mathcal O_S$ is surjective. 
Since the problem is local, by shrinking $Y$, 
we may assume that $f(S)=P\in Y$. 
If $S$ is connected, 
then $f_*\mathcal O_{S}\simeq \mathcal O_P$ 
and $\mathcal O_Y\to f_*\mathcal O_S$ is surjective. 
Therefore, it is sufficient to prove that 
$S$ is connected. 
We define a reduced 
divisor $D$ by 
$$
S+D=\Supp (f^*P). 
$$ 
If $D=0$, then $S$ is connected since 
$S=\Supp (f^*P)$. 
Therefore, 
we assume that $D\ne 0$. 
Then, there exists an irreducible curve 
$E\subset\Supp D$ such that $E\cap S\neq 0.$ 
We see $(K_X+\{\Delta\})\cdot E<0.$ 
Thus, we obtain a birational morphism $q:X\to X'$ 
such that $\Ex(q)=E.$ 
Let $\Delta':=q_*\Delta$. 
By Lemma~\ref{2.4}, if $\Supp\llcorner\Delta'\lrcorner$ is connected, 
then so is  $\Supp\llcorner\Delta\lrcorner$. 
We can repeat this argument and 
we obtain a projective morphisms 
$$f:X\overset{\tilde q}\to X''\overset{f''}\to Y$$ where 
$\tilde q$ is a birational morphism such that $\Ex(\tilde q)=\Supp D$. 
Let $\Delta'':=\tilde q_*\Delta$. 
It is sufficient to show that 
$\Supp\llcorner\Delta''\lrcorner$ is connected. 
This follows from $\Supp\llcorner\Delta''\lrcorner=\Supp(f''^*P).$ 
\end{proof}

In Proposition~\ref{2.1}, 
the most complicated case is the Mori fiber space 
to a curve. 
Thus we investigate this case in the following lemma. 

\begin{lem}\label{2.3}
Let $f':X'\to R$ be a projective surjective morphism 
between normal varieties such that $f'_*\mathcal O_{X'}=\mathcal O_R.$ 
Assume the following conditions. 
\begin{enumerate}
\item{$(X', \Delta')$ is a \Q-factorial lc surface 
such that $(X', \{\Delta'\})$ is klt.}
\item{$S':=\llcorner\Delta' \lrcorner\neq 0$.}
\item{$K_{X'}+\Delta'\equiv_{f'} 0$.}
\item{There is a 
$(K_{X'}+\{\Delta'\})$-negative 
extremal contraction $g':X'\to V$ over $R$ 
such that $\dim V=1.$}
\end{enumerate}
Then the $g'$-horizontal part $(S')^h$ of 
$S'$ satisfies 
one of the following assertions. 
\begin{enumerate}
\item[$(a)$]{$(S')^h=S_1'$, 
which is a prime divisor, and $[K(S_1'):K(V)]=2.$}
\item[$(b)$]{$(S')^h=S_1'$, 
which is a prime divisor, and $[K(S_1'):K(V)]=1.$}
\item[$(c)$]{$(S')^h=S_1'+S_2'$, 
where each $S_i'$ is a prime divisor, 
and $[K(S_i'):K(V)]=1.$}
\end{enumerate}
Furthermore, 
there is a \Q-Cartier \Q-divisor $D_V$ on $V$ such that 
$K_X+\Delta=g'^*(D_V)$.\\
In the case $(b)$, 
$f'_*\mathcal O_{S'}=\mathcal O_{f'(S')}$. 
\end{lem}

\begin{proof}
The assumption 
$(3)$ means $K_{X'}+\Delta'\equiv _{g'} 0.$ 
Thus, by $(4)$, 
$\llcorner\Delta'\lrcorner$ is $g'$-ample. 
We see $(S')^h\neq 0.$ 

We prove that 
general fibers of $g':X'\to V$ are $\mathbb{P}^1$. 
The dimension of every fiber is one. 
Since $\dim V=1$ and $f'_* \mathcal O_{X'}=\mathcal O_V$, 
the field extension 
$K(X')/K(V)$ is algebraically closed and separable 
(cf. \cite[Lemma~7.2]{Badescu}). 
Therefore general fibers are geometrically integral. 
Let $F$ be a general fiber of $g'$, that is, 
$F$ is a fiber which is 
a proper integral curve such that 
$F\cap \Sing(X)=\emptyset$. 
The adjunction formula implies 
$$(K_{X'}+F)\cdot F=K_{X'}\cdot F=
-\Delta'\cdot F\leq -(S')^h\cdot F<0. $$
This means $F\simeq \mathbb P^1.$

By $(K_{X'}+F)\cdot F=-2$, we have 
$(S')^h\cdot F\leq 2$ for a general fiber $F$. 
Therefore one of $(a)$, $(b)$ and $(c)$ holds. 
By the abundance theorem (\cite[Theorem~6.7]{Tanaka1}), we see 
$K_{X'}+\Delta'\sim_{\mathbb Q, g'}0$. 
This means $m(K_{X'}+\Delta')=g'^*(D)$ for some 
integer $m$ and some \Z-divisor $D$ on $V$. 
We define a \Q-divisor $D_V$ by $D=mD_V.$ 

Assume $(b)$ and let us prove $f'_*\mathcal O_{S'}=\mathcal O_{f'(S')}$. 
Since $\dim V=1$, we have $\dim R=0$ or $\dim R=1.$ 
Assume $\dim R=0$. 
It is sufficient to prove that $S'$ is connected. 
This holds because all of the fibers of $g'$ 
are irreducible and $(S')^h\neq 0$. 
Assume $\dim R=1$. 
Then, we see $f'(S')=V\simeq R.$ 
Since $S_1'$ and $R$ are birational, 
the morphism $f|_{S_1'}:S_1'\to R$ 
is an isomorphism. 
We can write 
$$S'=S_1'+F_1+\cdots+F_r$$
where each $F_i$ is the reduced subscheme whose 
support is a fiber of $g'$. 

We prove $f'_*\mathcal O_{S'}=\mathcal O_R$ by the induction on $r$. 
If $r=0$, then the assertion 
follows from $S_1'\simeq R$. 
Assume $r>0.$ 
Consider the exact sequence: 
$$0\to \mathcal O_{S'}\to
\mathcal O_{S'-F_r}\oplus\mathcal O_{F_r}\to 
\mathcal O_{(S'-F_r)\cap F_r}\to 0.$$
The last map defined by the difference. 
Note that the last term 
is the scheme-theoretic intersection. 
It is easy to see that 
$(S'-F_r)\cap F_r \simeq S_1'\cap F_r.$ 
Then $(S'-F_r)\cap F_r$ is reduced because 
$S_1'\simeq R$. 
Consider the push-forward of the above exact sequence: 
$$0\to f'_*\mathcal O_{S'}\to 
f'_*\mathcal O_{S'-F_r}\oplus f'_*\mathcal O_{F_r}\to 
f'_*\mathcal O_{(S'-F_r)\cap F_r}\to R^1f'_*\mathcal O_{S'}.$$
We see $R^1f'_*\mathcal O_{X'}=0$ by \cite[Theorem~2.12]{Tanaka2}. 
This implies $R^1f'_*\mathcal O_{S'}=0.$ 
Since $F_r$ and $(S'-F_r)\cap F_r$ are reduced, 
we have $f'_*\mathcal O_{F_r}\simeq 
f'_*\mathcal O_{(S'-F_r)\cap F_r}\simeq \mathcal O_{f'(F_r)}.$ 
This means $f'_*\mathcal O_{S'}\to
f'_*\mathcal O_{S'-F_r}$ is an isomorphism. 
By the induction hypothesis, 
we obtain $f'_*\mathcal O_{S'}\simeq
f'_*\mathcal O_{S'-F_r}\simeq \mathcal O_R.$
\end{proof}

\begin{rem}
In the last argument in the above proof, 
we use the following fact. 
Let $A$ be a ring and 
let $M, N, L$ and $P$ are $A$-modules. 
Assume the exact sequence 
$$0\to M\overset{(\varphi_1, \varphi_2)}\to 
N\oplus L\overset{\psi-\theta}\to P\to 0. $$
If $\theta:L\to P$ is an isomorphism, then 
$\varphi_1:M\to N$ is also an isomorphism. 
\end{rem}

We can prove the following main result 
in this section.

\begin{prop}\label{2.1}
Let $(X, \Delta)$ be an irreducible dlt surface. 
Let $f:X\to R$ be a projective surjective 
morphism such that $f_*\mathcal O_X=\mathcal O_R.$ 
Assume that $S:=\llcorner\Delta\lrcorner\neq 0$ and 
let $T:=f(S)$. 
If $K_X+\Delta\equiv_f 0$, then 
one of the following assertions holds. 
\begin{enumerate}
\item{$f_*\mathcal O_S=\mathcal O_T$.}
\item
{$f_*\mathcal O_S\neq\mathcal O_T$. 
There exist a projective surjective $R$-morphism 
$g:X\to V$ to a smooth curve $V$ and 
a \Q-divisor $D_V$ on $V$ such that 
$g_*\mathcal O_X=\mathcal O_V$ and that 
$K_X+\Delta=g^*(D_V)$ as \Q-divisors. 
Every connected component of $S$ 
intersects the $g$-horizontal part $S^h$ of $S$. 
Moreover, the $g$-horizontal part $S^h$ satisfies one of 
the following assertions. 
\begin{enumerate}
\item[(2.1s)]{$S^h=S_1$, which is a prime divisor, 
and $[K(S_1):K(V)]=2.$ 
This field extension is separable. }
\item[(2.1i)]{$S^h=S_1$, which is a prime divisor, 
and $[K(S_1):K(V)]=2.$ 
This field extension is purely inseparable. }
\item[(2.2)]{$S^h=S_1+S_2$, 
where $S_i$ is a prime divisor, and 
$g|_{S_i}:S_i\to V$ is an isomorphism 
for $i=1, 2$.}
\end{enumerate}}
\end{enumerate}
\end{prop}

\begin{proof}
If $f$ is birational, then 
Proposition~\ref{4.4} implies $(1)$. 
Thus we may assume that $\dim R<\dim X.$ 
We run a $(K_X+\{\Delta\})$-MMP 
on $X$ over $R$. 
The end result is a proper birational morphism 
$q:X\to X'$ over $R$. 
Let $f':X'\to R$ be the induced morphism. 
Since $K_X+\Delta\equiv_f 0,$ 
we obtain $K_{X'}+\Delta'\equiv_{f'} 0$ 
where $\Delta':=q_*\Delta.$ 
Let $S':=\llcorner \Delta'\lrcorner$. 
Then it is easy to see that 
$(X', \Delta')$ is a \Q-factorial lc pair and 
$(X, \{\Delta'\})$ is klt. 
\setcounter{step}{0}
\begin{step}
Assume that 
$(X', \{\Delta'\})$ is a minimal model over $R$. 
Then $K_{X'}+\{\Delta'\}$ 
is $f'$-nef and $K_{X'}+\Delta'\equiv_{f'} 0.$ 
So $-\llcorner \Delta'\lrcorner$ is $f'$-nef. 
If $\dim R=0$, then 
$\llcorner\Delta'\lrcorner=0$ 
because $X'$ is projective. 
Lemma~\ref{2.4} implies $\llcorner\Delta\lrcorner=0.$ 
This case is excluded. 
Assume $\dim R=1.$ 
Since $-\llcorner \Delta'\lrcorner$ is $f'$-nef, 
we see $f'(\llcorner \Delta'\lrcorner)\subsetneq R.$ 
Therefore, by Proposition~\ref{4.4}, we obtain $(1).$ 
\end{step}

\begin{step}
Assume that there exists a Mori fiber space structure 
$g':X'\to V$ over $R$. 
Let 
$$g:X\overset{q}\to X'\overset{g'}\to V.$$
Then 
$-(K_{X'}+\{\Delta'\})$ is $g'$-ample. 
Note that, if $\dim V=1$, then 
we can apply Lemma~\ref{2.3} and every connected component of $S$ 
intersects $S^h$ by Lemma~\ref{2.4}. 

First, assume that $\dim R=0.$ 
If $\llcorner \Delta'\lrcorner$ is connected, 
then we have $(1)$ by Lemma~\ref{2.4}. 
Thus we may assume that 
$\llcorner \Delta'\lrcorner$ is not connected. 

We show $\dim V=1$. 
Assume $\dim V=0$. 
Then $\llcorner \Delta'\lrcorner$ is ample. 
Thus its suitable multiple is an 
effective ample Cartier divisor. 
This must be connected 
by the Serre vanishing theorem. 
This case is excluded. 

Thus we can apply Lemma~\ref{2.3}. 
Since 
all of the fibers of 
the Mori fiber space $g':X'\to V$ are irreducible, 
we see 
$\llcorner\Delta'\lrcorner=S'_1+S'_2$. 
This implies $(2.2)$. 

Second, assume that $\dim R=1.$ 
Then we have $\dim V=1.$ 
Note that $T=R\simeq V$. 
We can apply Lemma~\ref{2.3}. 
Thus we obtain $(a), (b)$ or $(c)$ of Lemma~\ref{2.3}. 
If $(a)$ or $(c)$ holds, then 
$(2)$ holds. 
Thus we may assume that 
$(b)$ of Lemma~\ref{2.3} holds. 
We have $f'_*\mathcal O_{S'}=\mathcal O_T.$ 
Lemma~\ref{2.4} implies $q(S)=S'$. 
By Proposition~\ref{1.3}, 
we have $f_*\mathcal O_S=f'_*\mathcal O_{S'}=\mathcal O_T.$
\end{step}
\end{proof}

\begin{ex}
Let ${\rm char}\,k=2.$ 
Then, there exists a projective dlt surface $(X, \Delta)$ 
and smooth projective curve $R$ which satisfy 
{\rm Proposition~\ref{2.1}$(2.1{\rm i})$}. 
\end{ex}

\begin{proof}[Construction]
Let $X_0:=\mathbb A^2$ and 
let $C_0:=\{(x, y)\in\mathbb A^2\,|\,x=y^2\}$. 
Note that the restriction of the first projection 
to $C_0$ is purely inseparable of degree two. 
Let $X_0\subset X:=\mathbb P^1\times\mathbb P^1$ 
be the natural open immersion and 
let $C$ be the closure of $C_0$ in $X$. 
Let $g:X\to \mathbb P^1=:V=:R$ 
be the first projection. 
It is easy to see that $C$ is smooth 
and $K_X+C\sim g^*\mathcal O_{\mathbb P^1}(-1).$ 
Thus, we see that $(X, \Delta:=C)$ is dlt and 
that $(X, \Delta)$ satisfies Proposition~\ref{2.1}$(2.1{\rm i})$. 
\end{proof}

\section{Normalization of nodes}

In this section, we calculate the normalization of nodal singularities 
to reduce problems for slc varieties to ones for dlt varieties. 
The main theorem of this section is Theorem~\ref{theorem-node}. 
In this section, we do not work over a field and 
we treat noetherian or excellent schemes. 

First we recall the definition of the nodal singularities in the sense of \cite[1.41]{Kollar}. 

\begin{dfn}\label{def-node}
Let $(R, \mathfrak m)$ be a noetherian local ring. 
We say $R$ has a {\em node} (or $R$ is {\em nodal}) if 
there exists an isomorphism $R\simeq S/(f)$ where 
$(S, \mathfrak l)$ is a two-dimensional regular local ring such that 
$f\in\mathfrak l^2$ and that 
$f$ is not a square in $\mathfrak l^2/\mathfrak l^3$. 
\end{dfn}

We mainly use the following notations. 

\begin{nota}\label{nota-node}
Let $(R, \mathfrak m)$ be a nodal noetherian local ring. 
By definition, we can write $R\simeq S/(f)$ where 
$(S, \mathfrak l)$ is a two-dimensional regular local ring such that 
$f\in\mathfrak l^2$ and that 
$f$ is not a square in $\mathfrak l^2/\mathfrak l^3$. 
Take a generator $\mathfrak l=(x, y)$. 
We can write 
$$f=ax^2+bxy+cy^2+g$$
where $a, b, c\in \{0\}\cup S^{\times}$ and $g\in \mathfrak l^3$. 
We set $\overline x:=x+(f)\in R/(f)$ and $\overline y:=y+(f)\in R/(f)$. 
\end{nota}

\begin{rem}\label{rem-node}
We use the same notations as Notation~\ref{nota-node}. 
We show that we may assume 
$$c\in S^{\times}$$ 
by replacing a generator $\{x, y\}$ of $\mathfrak l$. 
If $c \in S^{\times}$, then there is nothing to show. 
If $a\in S^{\times}$, then we exchange $x$ and $y$. 
Since $a, c\in \{0\}\cup S^{\times}$, we assume $a=c=0$. 
By $f\not\in\mathfrak l^3$, we see $b\not\in \mathfrak l$, that is, 
$b\in S^{\times}$. 
Taking another generator $X:=x-y, Y:=y$ of $\mathfrak l=(x, y)=(X, Y)$, 
we obtain 
\begin{eqnarray*}
f
&=&bxy+g\\
&=&b(X+Y)Y+g\\
&=&bXY+bY^2+g.
\end{eqnarray*}
By $b\in S^{\times}$, we may assume $c\in S^{\times}$. 
\end{rem}

We calculate the normalization of nodes. 
We divide the proof into the following two cases: 
$R$ is an integral domain or not. 
In Lemma~\ref{lemma-not-integral}, 
we treat the case where $R$ is not an integral domain. 
In Lemma~\ref{lemma-integral}, 
we treat the case where $R$ is an integral domain.

\begin{lem}\label{lemma-not-integral}
Let $(R, \mathfrak m)$ be a nodal noetherian local ring. 
We use the same notations as {\rm Notation~\ref{nota-node}}. 
Assume that 
$R$ is not an integral domain. 
Then the following assertions hold. 
\begin{enumerate}
\item{$f$ has a decomposition $f=l_1l_2$ with $l_1, l_2 \in S$ which satisfies the following properties. 
\begin{itemize}
\item{$l_1S\neq l_2S$.}
\item{For each $i$, $l_i\in \mathfrak l\setminus \mathfrak l^2$.}
\item{For each $i$, $l_i$ is a prime element of $S$, that is, $l_iS$ is a prime ideal.}
\end{itemize}
}
\item{$l_1$ and $l_2$ satisfies $\mathfrak l=(l_1, l_2)$.}
\item{For each $i$, $S/(l_i)$ is regular. }
\item{The natural homomorphism 
$$\nu:R=S/(f)=S/(l_1l_2) \to S/(l_1)\times S/(l_2)=:T$$
is the normalization.}
\item{$\mathfrak m$ is the conductor of the normalization $\nu:R\hookrightarrow T$, that is, 
$$\mathfrak m=\{r\in R\,|\, rT \subset R\}.$$}
\item{The normalization $\nu:R\hookrightarrow T$ induces 
$$\theta: k(\mathfrak m)=R/\mathfrak m \to T/\mathfrak mT\simeq k(\mathfrak m)\times k(\mathfrak m),$$
where $p_i\circ \theta$ is the identity map for the projection $p_i$ to the $i$-th factor.}
\end{enumerate}
\end{lem}

\begin{proof}
(1) 
Since $S$ is a unique factorization domain, 
we obtain a decomposition of $f$ into prime elements: 
$$f=ul_1^{n_1}\cdots l_r^{n_r}$$
where $u\in S^{\times}$, $n_i\in\mathbb Z_{>0}$ and $l_i$ is a prime element of $S$. 
In particular, $l_i\in \mathfrak l$. 
Then, $f\not\in\mathfrak l^3$ implies $n_1+\cdots+n_r\leq 2$. 
Since $n_1+\cdots+n_r=1$ implies that $R$ is an integral domain, 
we see $n_1+\cdots+n_r=2$. 
Thus, we obtain one of the following two cases: $f=ul_1^2$ or $f=ul_1l_2$ where $l_1S\neq l_2S$. 
By $f\not\in\mathfrak l^3$ and $l_i\in \mathfrak l$, 
we see $l_i\not\in \mathfrak l^2$. 
Then, it is enough to show that the case $f=ul_1^2$ does not occur. 
Suppose $f=ul_1^2$. 
We can write 
$$l_1=\alpha x+\beta y+h$$
where $\alpha, \beta\in \{0\}\cup S^{\times}$ and $h\in \mathfrak l^2$. 
We obtain 
$$f=ul_1^2=u(\alpha x+\beta y+h)^2=u(\alpha x+\beta y)^2+({\rm an\,\,element\,\,of\,\,}\mathfrak l^3).$$
By replacing $f$ with $u^{-1}f$, 
this contradicts the definition of nodes: Definition~\ref{def-node}.

(2) 
Since $R$ is nodal, 
$(l_1, l_2)$ generates $\mathfrak l/\mathfrak l^2$. 
Then Nakayama's lemma implies the assertion. 

(3) 
The assertion follows from (2). 

(4) 
The assertion follows from (3). 

(5) 
Let $I\subset R$ be the conductor. 
The inclusion $\mathfrak m \supset I$ is clear. 
We show the inverse inclusion $(l_1, l_2)=\mathfrak m \subset I$. 
By the symmetry, it suffices to prove $l_1\in I$. 
Take $\xi=(s_1+(l_1), s_2+(l_2))\in S/(l_1)\times S/(l_2)=T$. 
Then, we obtain $l_1\xi=(0+(l_1), l_1s_2+(l_2))$. 
Therefore, $l_1\xi=\nu(l_1s_2)$. 
This is what we want to show. 

(6) 
By $\nu(l_1+l_2)=(l_2+(l_1), l_1+(l_2))$, 
we see $\mathfrak mT=\mathfrak m/(l_1)\times \mathfrak m/(l_2)$.  
This implies the assertion. 
\end{proof}

\begin{lem}\label{lemma-integral}
Let $(R, \mathfrak m)$ be a nodal noetherian local ring. 
We use the same notations as {\rm Notation~\ref{nota-node}}. 
Suppose $c \in S^{\times}$ $(${\rm cf.} {\rm Remark~\ref{rem-node}}$)$. 
Assume that $R$ is an integral domain. 
Consider the following natural injective ring homomorphism  
$$\varphi:R \hookrightarrow R\left[\frac{\overline y}{\overline x}\right]=:T.$$
Then the following assertions hold. 
\begin{enumerate}
\item{The ring homomorphism 
$\theta:S[\frac{y}{x}]/(f/x^2)\to R[\frac{\overline y}{\overline x}]=T,\,\,\, \frac{y}{x}\mapsto \frac{\overline y}{\overline x}$ 
is an isomorphism.}
\item{$T$ is a regular ring. }
\item{One of the following assertions holds. 
\begin{enumerate}
\item[(a)]{$T/\mathfrak mT\simeq k(\mathfrak m)\times k(\mathfrak m)$ and 
the composition homomorphism  
$$k(\mathfrak m)=R/\mathfrak m \to T/\mathfrak mT\simeq k(\mathfrak m)\times k(\mathfrak m) \overset{p_i}\to k(\mathfrak m)$$
is the identity map for $i=1, 2$ 
where $p_i$ is the projection to the $i$-th factor.}
\item[(b)]{$T/\mathfrak mT$ is a field and the natural homomorphism  
$$k(\mathfrak m)=R/\mathfrak m \to T/\mathfrak mT$$
is a field extension with $[T/\mathfrak mT:k(\mathfrak m)]=2.$}
\end{enumerate}}
\item{The equation $(\frac{\overline y}{\overline x})^2+r_1\frac{\overline y}{\overline x}+r_2=0$ 
holds in $ R[\frac{\overline y}{\overline x}]=T$ for some $r_1, r_2\in R$. 
In particular, $T$ is a finitely generated $R$-module.}
\item{$T$ is the integral closure of $R$ in the quotient field $K(R)$.}
\item{The maximal ideal $\mathfrak m$ is the conductor of the normalization, that is, 
$\mathfrak m=\{r\in R\,|\, rT \subset R\}.$}
\end{enumerate}
\end{lem}

\begin{proof}
We use the same notations as Notation~\ref{nota-node}. 

(1)  
Set $z:=\frac{y}{x}\in K(S)$. 
Let us check $f/x^2\in S[\frac{y}{x}]=S[z]$. 
Since $f\in\mathfrak l^2=(x, y)^2$, 
we can write $f=\alpha x^2+\beta xy+\gamma y^2$ 
for some $\alpha, \beta, \gamma\in S$. 
Then we see $f/x^2 \in S[z]$ by the following calculation: 
$$f=\alpha x^2+\beta xy+\gamma y^2=\alpha x^2+\beta x(xz)+\gamma (xz)^2=x^2(\alpha +\beta z+\gamma z^2).$$

Consider the natural homomorphism 
\begin{eqnarray*}
\theta:S\left[\frac{y}{x}\right]/(f/x^2) &\to& R\left[\frac{\overline y}{\overline x}\right]\\
\frac{y}{x}&\mapsto& \frac{\overline y}{\overline x}.\\
\end{eqnarray*}
We prove that $\theta$ is an isomorphism. 
For the time being, we show this assuming the following two assertions. 
\begin{enumerate}
\item[(A)]{The $S$-algebra homomorphism $S/(f) \to S[\frac{y}{x}]/(f/x^2)$ is injective. }
\item[(B)]{$S[\frac{y}{x}]/(f/x^2)$ is an integral domain.}
\end{enumerate}
Consider the following commutative diagram of $S$-algebras: 
$$\begin{CD}
S/(f) @= R\\
@VV{\rm injective}V @VV{\rm injective}V\\
S[\frac{y}{x}]/(f/x^2) @>\theta >> R[\frac{\overline y}{\overline x}].
\end{CD}$$
Note that 
$R[\frac{\overline y}{\overline x}] \subset K(R)=K(S/(f)) \subset K(S[\frac{y}{x}]/(f/x^2))$. 
All of the four rings in the above diagram are contained in 
the quotient field $K(S[\frac{y}{x}]/(f/x^2))$. 
In $K(S[\frac{y}{x}]/(f/x^2))$,  
the element $\frac{y}{x}+(f/x^2)\in S[\frac{y}{x}]/(f/x^2)$ is the same as 
$\frac{\overline y}{\overline x}\in R[\frac{\overline y}{\overline x}]$. 
Therefore we obtain 
$$S\left[\frac{y}{x}\right]/(f/x^2)=R\left[\frac{\overline y}{\overline x}\right].$$

(A) 
We show that the natural map $S/(f) \to S[\frac{y}{x}]/(f/x^2)$ is injective. 
For this, consider the following natural composition map 
$$\psi: S \to S\left[\frac{y}{x}\right] \to S\left[\frac{y}{x}\right]/(f/x^2)$$
and we show $\Ker(\psi)=fS$. 
The inclusion $\Ker(\psi)\supset fS$ is obvious. 
Let us prove the inverse inclusion $\Ker(\psi)\subset fS$. 
Take an element $s\in S$ such that $\psi(s)=0$, that is, 
$s\in \frac{f}{x^2}S[\frac{y}{x}]$. 
We have 
$$s=\frac{f}{x^2}\left(t_0+t_1\frac{y}{x}+\cdots+t_m\frac{y^m}{x^m}\right)$$
where $t_i \in S$. 
Let us show that we can assume $m=0$. 
Assume $m\geq 1$. 
Moreover assume $t_m\in xS$, that is, $t_m=x\tilde t_m$ with $\tilde t_m \in S$. 
Then, by the following calculation: 
$$t_m\frac{y^m}{x^m}=x\tilde t_m\frac{y^m}{x^m}=y\tilde t_m\frac{y^{m-1}}{x^{m-1}},$$
we obtain another expression: $s=\frac{f}{x^2}(t_0+\cdots+t_{m-2}(\frac{y}{x})^{m-2}+t_{m-1}'\frac{y^{m-1}}{x^{m-1}})$ 
for some $t_{m-1}'\in S$. 
Thus, we assume $m\geq 1$ and $t_m\not\in xS$. 
Taking the multiplication with $x^{m+2}$, we obtain 
$$sx^{m+2}=f(t_0x^m+\cdots+t_{m-1}xy^{m-1}+t_my^m).$$
This implies $ft_my^m \in xS$. 
But, both the elements $t_m$ and $y$ are not in $xS$. 
Since $xS$ is a prime ideal, we obtain $f\in xS$. 
Then we can write $f=xg$ with $g\in S$. 
$f\in \mathfrak l^2$ implies $g\in\mathfrak l$. 
Therefore, $f$ is not a prime element, 
which contradicts that $R$ is an integral domain. 
Therefore, we may assume $m=0$ and we obtain 
$$s=\frac{f}{x^2}t_0.$$
Since $f\not \in xS$, we see $t_0\in xS$. 
Repeating this, we see $t_0\in x^2S$, which implies $s\in fS$. 
This is what we want to show. 

(B) 
First we prove that $S[\frac{y}{x}]$ is a unique factorization domain. 
We see that $xS[\frac{y}{x}]$ is a prime ideal because 
$$S\left[\frac{y}{x}\right]/xS\left[\frac{y}{x}\right]\simeq S[Z]/(x, xZ-y) 
\simeq (S/(x, y))[Z]$$
is an integral domain. 
By Nagata's criterion (\cite[Lemma~1]{Nagata}), $S[\frac{y}{x}]$ is a unique factorization domain 
if so is 
$$\left(S\left[\frac{y}{x}\right]\right)\left[\frac{1}{x}\right]=S\left[\frac{1}{x}\right].$$
This ring $S\left[\frac{1}{x}\right]$ is a unique factorization domain because so is $S$.

We show that $S[\frac{y}{x}]/(f/x^2)$ is an integral domain. 
Since $S[\frac{y}{x}]$ is a unique factorization domain, 
let us check that $f/x^2$ is a prime element. 
Assume that there exists a decomposition 
$$\frac{f}{x^2}=\left(s_0+s_1\frac{y}{x}+\cdots+s_k\frac{y^k}{x^k}\right)
\left(t_0+t_1\frac{y}{x}+\cdots+t_l\frac{y^l}{x^l}\right)$$
where $s_i, t_j\in S$ and both the factors in the right hand side 
are not in $(S[\frac{y}{x}])^{\times}$. 
We may assume that, if $k\geq 1$ (resp. $l\geq 1$), 
then $s_k$ (resp. $t_l$) is not in $xS$. 
We show that $k=0$ or $l=0$ holds. 
Assume $k\geq 1$ and $l\geq 1$. 
We consider the following two cases: 
$k=l=1$ and $k+l\geq 3$. 
If $k=l=1$, then we obtain $f=(s_0x+s_1y)(t_0x+t_1y)$. 
This contradicts that $f$ is a prime element. 
If $k+l\geq 3$, then taking the multiplication with $x^{k+l}$, we see $s_kt_ly^{k+l}\in xS$. 
By $s_k\not \in xS$ and $t_l\not\in xS$, 
we have $y^{k+l}\in xS$, which is a contradiction. 
Therefore, $k=0$ or $l=0$ holds. 
By the symmetry, we may assume $l=0$ and we obtain 
$$\frac{f}{x^2}=\left(s_0+s_1\frac{y}{x}+\cdots+s_k\frac{y^k}{x^k}\right)t_0.$$
If $k\geq 1$ and $t_0\in xS$, then, for $t_0=xt_0'$, 
we obtain another expression: 
$\frac{f}{x^2}=(xs_0+\cdots+xs_{k-1}(\frac{y}{x})^{k-1}+ys_k(\frac{y}{x})^{k-1})t_0'$. 
Thus we may assume that $k=0$ or $t_0\not\in xS$ holds. 
If $k=0$, then we obtain the following contradiction: 
$f=x^2s_0t_0.$ 
Assume $t_0\not\in xS$. 
Taking the multiplication with $x^k$, we see $k\leq 2$. 
This implies $f=(s_0x^2+s_1xy+s_2y^2)t_0$. 
Since $s_0x^2+s_1xy+s_2y^2\in\mathfrak m$ and $f\in S$ is a prime element, 
we have $t_0 \in S^{\times} \subset (S[\frac{y}{x}])^{\times}$. 
This is a contradiction.

(2) 
Set $z=\frac{y}{x}$. 
First, we calculate the ring $(S[\frac{y}{x}]/(f/x^2))/(\overline x)$. 
The element $f/x^2$ can be written 
$$\frac{f}{x^2}=\frac{ax^2+bxy+cy^2+g}{x^2}=\frac{ax^2+bx(xz)+c(xz)^2+x^3\tilde g}{x^2}=a+bz+cz^2+x\tilde g$$
for some $\tilde g\in S[z]$. 
Here, since $(S, (x, y))$ is a regular local ring, 
we can check that the homomorphism 
$$S[Z]/(xZ-y) \to S\left[\frac{y}{x}\right],\,\,\, Z \mapsto \frac{y}{x}$$
is an isomorphism. 
Then, we see 
\begin{eqnarray*}
(S\left[\frac{y}{x}\right]/(f/x^2))/(\overline x)
&\simeq& S\left[\frac{y}{x}\right]/((f/x^2)+(x))\\
&\simeq& S[Z]/(xZ-y, a+bZ+cZ^2, x) \\
&\simeq& k(\mathfrak m)[Z]/(\overline a+\overline bZ+\overline cZ^2).
\end{eqnarray*}

Fix a maximal ideal $\mathfrak n$ of $S[\frac{y}{x}]/(f/x^2)$ 
and we show that the local ring $(S[\frac{y}{x}]/(f/x^2))_{\mathfrak n}$ is regular. 

We show $\overline x \in \mathfrak n$. 
Assume $\overline x \not\in \mathfrak n$. 
Then $\mathfrak n$ corresponds to a maximal ideal 
of $R[\frac{\overline y}{\overline x}][\frac{1}{\overline x}]=R[\frac{1}{\overline x}]$, that is, 
$\mathfrak n=(\mathfrak nR[\frac{\overline y}{\overline x}][\frac{1}{\overline x}]) \cap R[\frac{\overline y}{\overline x}]$. 
Since $(R, \mathfrak{m})$ is one dimensional local integral domain and $\overline x\in \mathfrak m$, 
$R[\frac{1}{\overline x}]$ is a field. 
It implies $\mathfrak n=(0)$. 
Then $S[\frac{y}{x}]/(f/x^2)$ is a field. 
On the other hand, 
by the above isomorphism 
$$(S\left[\frac{y}{x}\right]/(f/x^2))/(\overline x)\simeq k(\mathfrak m)[Z]/(\overline a+\overline bZ+\overline cZ^2)$$
and $\overline c\neq 0$, 
there exists a non-zero ideal $(\overline x)$ of $S[\frac{y}{x}]/(f/x^2)$. 
Thus $S[\frac{y}{x}]/(f/x^2)$ is not a field and we obtain a contradiction.

Therefore, $\overline x\in \mathfrak n$. 
To show that 
the local ring $(S[\frac{y}{x}]/(f/x^2))_{\mathfrak n}$ is regular, 
it is enough to prove that the ring 
$$(S\left[\frac{y}{x}\right]/(f/x^2))/(\overline x)\simeq k(\mathfrak m)[Z]/(\overline a+\overline bZ+\overline cZ^2)$$
is regular. 
If $\overline a+\overline bZ+\overline cZ^2$ is irreducible over $k(\mathfrak m)$, then 
the ring $k(\mathfrak m)[Z]/(\overline a+\overline bZ+\overline cZ^2)$ is a field. 
Assume that $\overline a+\overline bZ+\overline cZ^2$ is not irreducible over $k(\mathfrak m)$. 
We have $\overline c\neq 0$. 
There are $\alpha, \beta\in R$ such that 
$$\overline a+\overline bZ+\overline cZ^2=\overline c(Z+\overline \alpha)(Z+\overline \beta).$$
Since $R$ is nodal, we see $\overline \alpha\neq \overline \beta$. 
Therefore, 
$$(S\left[\frac{y}{x}\right]/(f/x^2))/(\overline{x}) \simeq k(\mathfrak m)[Z]/(\overline a+\overline bZ+\overline cZ^2)\simeq k(\mathfrak m) \times k(\mathfrak m).$$
This is what we want to show.

(3) 
Let us calculate $T/\mathfrak mT$. 
By 
$$\mathfrak mT=\mathfrak mR\left[\frac{\overline y}{\overline x}\right]
=(\overline x, \overline y)R\left[\frac{\overline y}{\overline x}\right]
=\overline xR\left[\frac{\overline y}{\overline x}\right],$$
we obtain $T/\mathfrak mT\simeq (S[\frac{y}{x}]/(f/x^2))/(\overline x)$. 
By the proof of (2), we obtain 
$$(S\left[\frac{y}{x}\right]/(f/x^2))/(\overline x)\simeq k(\mathfrak m)[Z]/(\overline a+\overline bZ+\overline cZ^2).$$
If $\overline a+\overline bZ+\overline cZ^2$ is irreducible, then we obtain (b). 
Assume that $\overline a+\overline bZ+\overline cZ^2$ is not irreducible. 
Then, we can write 
$$\overline a+\overline bZ+\overline cZ^2=\overline c(Z+\overline \alpha)(Z+\overline \beta).$$
Since $R$ is nodal, we see $\overline \alpha\neq \overline \beta$. 
This implies (a).

(4) 
By Notation~\ref{nota-node}, 
we have 
$$f=ax^2+bxy+cy^2+g$$
where $a, b, c\in\{0\}\cup S^{\times}$ and $g\in \mathfrak l^3=(x, y)^3$. 
Moreover, we have $c\in S^{\times}$. 
For some $\alpha, \beta, \gamma, \delta\in S$, we obtain 
$$f=ax^2+bxy+cy^2+\alpha x^3+\beta x^2y+\gamma xy^2+\delta y^3,$$
which implies 
\begin{eqnarray*}
\frac{f}{x^2}
&=&a+b\frac{y}{x}+c\left(\frac{y}{x}\right)^2+\alpha x+\beta y+\gamma y\frac{y}{x}+\delta y\left(\frac{y}{x}\right)^2\\
&=&(c+\delta y)\left(\frac{y}{x}\right)^2+(b+\gamma y)\frac{y}{x}+(a+\alpha x+\beta y).
\end{eqnarray*}
By $c\in S^{\times}$ and $\delta y\in \mathfrak l$, we see $c+\delta y\in S^{\times}$. 
Therefore the assertion follows from (1).

(5) 
The assertion follows from (2) and (4). 

(6) 
Let $I:=\{r\in R\,|\, rT \subset R\}$ be the conductor ideal. 
By this definition, $I$ is an ideal of $R$. 
Note that $I$ is also an ideal of $T$. 
Since $R\neq T$, we obtain $1\not\in I$. 
In particular, $I\subset \mathfrak m$. 
Let us show $I\supset \mathfrak m$. 
By (4), we obtain 
$$T=R\left[\frac{\overline y}{\overline x}\right]=R+R\frac{\overline y}{\overline x}.$$
This implies $\overline{x}T\subset R$. 
Thus, $\overline{x}\in I$. 
Since $I$ is an ideal of $T=R[\frac{\overline y}{\overline x}]$, we see 
$\overline{y}=\overline x \frac{\overline y}{\overline x}\in I.$ 
Therefore, $I\supset \overline xR+\overline yR=\mathfrak m$. 
\end{proof}

We say a scheme $X$ is {\em excellent} if $X$ is covered 
by open affine schemes whose corresponding rings are excellent.

Combining Lemma~\ref{lemma-not-integral} and Lemma~\ref{lemma-integral}, 
we obtain the following result.

\begin{prop}\label{prop-node}
Let $X$ be a quasi-compact excellent reduced scheme and 
let $\eta$ be a scheme-theoretic point whose local ring $\MO_{X, \eta}$ is nodal.
Let $S:=\overline{\{\eta\}}$ be the reduced scheme. 
Let $\nu:Y \to X$ be the normalization, 
$D\subset X$ the closed subscheme defined by the conductor and 
$C\subset Y$ its scheme-theoretic inverse image$:$ 
$$\begin{CD}
C:=\nu^{-1}(D) @>{\rm closed}>{\rm immersion}> Y \\
@VVV @VV\nu V\\
D @>{\rm closed}>{\rm immersion}> X.
\end{CD}$$
Then, there exists an open subset $\eta \in X'\subset X$ which satisfies the following properties. 
\begin{enumerate}
\item[(0)]{Set $Y':=\nu^{-1}(X')$, $D':=D\cap X'$, $C':=C\cap Y'$ and $S':=S\cap X'$. }
\item{$D'$ is reduced and $S'=D'$. In particular, 
$D'$ is an integral scheme.}
\item{$\nu|_{C'}:C' \to D'$ satisfies one of the following conditions 
\begin{itemize}
\item{$C'\simeq D'_1\amalg D'_2$ with $D_i'\simeq D$ and each morphism  
$$D'_i\hookrightarrow C' \overset{\nu|_{C'}}\to D'$$
are isomorphism.}
\item{$C'$ is an integral scheme and the field extension $K(C')/K(D')$ 
satisfies $[K(C'):K(D')]=2$.}
\end{itemize}}
\end{enumerate}
\end{prop}

\begin{proof}
We may assume 
$X=\Spec\, A, Y=\Spec\, B, D=\Spec\, A/I$ and $C=\Spec\, B/J$ where $I=J$. 
Let $S_{\eta}:=A\setminus \eta$ where we consider $\eta$ as a prime ideal of $A$. 
There are the following two cases. 
\begin{enumerate}
\item[($\alpha$)]{$\MO_{X, \eta}=A_{\eta}=S_{\eta}^{-1}A$ is not an integral domain.}
\item[($\beta$)]{$\MO_{X, \eta}=A_{\eta}=S_{\eta}^{-1}A$ is an integral domain.}
\end{enumerate}

($\alpha$) 
Assume that $S_{\eta}^{-1}A$ is not an integral domain. 
We can apply Lemma~\ref{lemma-not-integral} to $S_{\eta}^{-1}A$. 
Then, by shrinking $\eta\in\Spec\, A$, 
we obtain the following commutative diagram: 
$$
\begin{CD}
A @>>> A/\mathfrak p_1\times A/\mathfrak p_2\\
@VVV @VVV\\
S_{\eta}^{-1}A @>>> S_{\eta}^{-1}(A/\mathfrak p_1)\times S_{\eta}^{-1}(A/\mathfrak p_2),
\end{CD}$$
where $(0)=\mathfrak p_1\cap \mathfrak p_2$. 
Since $A$ is excellent, for each $i$, 
the regular locus $U_i$ of $\Spec\, A/\mathfrak p_i$ forms an open subset of $\Spec\, A/\mathfrak p_i$. 
Since $S_{\eta}^{-1}(A/\mathfrak p_i)$ is regular,  
we obtain $\eta\in U_i$. 
Therefore, by shrinking $\eta\in\Spec\, A$, 
we may assume that each $A/\mathfrak p_i$ is regular. 
In particular, the homomorphism 
$A \to A/\mathfrak p_1\times A/\mathfrak p_2$ coincides with the normalization. 
Since $S_{\eta}^{-1}(A/I)$ is reduced and $A$ is noetherian, 
we may assume that $A/I$ is reduced by shrinking $\Spec\, A$.  
This implies (1). 
We show (2). 
We have the induced homomorphism 
$$\theta_i:A/I \to (A/(I+\mathfrak p_1))\times (A/(I+\mathfrak p_2)) \to A/(I+\mathfrak p_i),$$
where the latter map is the projection to the $i$-th factor. 
By Lemma~\ref{lemma-not-integral}, $S_{\eta}^{-1}\theta_i$ is an isomorphism. 
Since $X=\Spec\, A$ is noetherian and the kernel and the cokernel of $\theta$ is a finitely generated $A$-modules, 
we obtain the assertion.

($\beta$) 
Assume that $S_{\eta}^{-1}A$ is an integral domain. 
We can apply Lemma~\ref{lemma-integral} to $S_{\eta}^{-1}A$. 
We obtain the following commutative diagram: 
$$
\begin{CD}
A @>\nu >> B\\
@VVV @VVV\\
S_{\eta}^{-1}A @>>> S_{\eta}^{-1}B.
\end{CD}$$
By Lemma~\ref{lemma-integral}, 
$S_{\eta}^{-1}(A/I)$ is reduced. 
This implies (1). 
By Lemma~\ref{lemma-integral}, 
there are the following two cases: 
\begin{enumerate}
\item[(a)]{$S_{\eta}^{-1}(B/J)\simeq S_{\eta}^{-1}(A/I)\times S_{\eta}^{-1}(A/I)$ and 
the composition homomorphism  
$$S_{\eta}^{-1}(A/I) \to S_{\eta}^{-1}(B/J)\simeq S_{\eta}^{-1}(A/I)\times S_{\eta}^{-1}(A/I) \overset{p_i}\to S_{\eta}^{-1}(A/I)$$
is the identity map for $i=1, 2$ 
where $p_i$ is the projection to the $i$-th factor.}
\item[(b)]{$S_{\eta}^{-1}(B/J)$ is a field and the natural homomorphism  
$$S_{\eta}^{-1}(A/I) \to S_{\eta}^{-1}(B/J)$$
is a field extension with $[S_{\eta}^{-1}(B/J):S_{\eta}^{-1}(A/I)]=2.$}
\end{enumerate}
For each case, we obtain (2) by a similar argument to ($\alpha$). 
\end{proof}

The following theorem is the main result in this section.

\begin{thm}\label{theorem-node}
Let $k$ be a field. 
Let $X$ be a pure-dimensional reduced separated scheme of finite type over $k$. 
Assume that $X$ is $S_2$ and, for every codimension one scheme-theoretic point $\eta\in X$, 
the local ring $\MO_{X, \eta}$ is regular or nodal. 
Let $\nu:Y \to X$ be the normalization, 
$D\subset X$ the closed subscheme defined by the conductor and 
$C\subset Y$ its scheme-theoretic inverse image: 
$$\begin{CD}
C:=\nu^{-1}(D) @>{\rm closed}>{\rm immersion}> Y \\
@VV\nu|_C V @VV\nu V\\
D @>{\rm closed}>{\rm immersion}> X.
\end{CD}$$
Let $L$ be an invertible sheaf on $X$ and fix $s\in H^0(Y, \nu^*L^{\otimes 2})$. 
Let $C=\bigcup C_i$ be the irreducible decomposition where each $C_i$ is an integral scheme. 
Assume the following conditions. 
\begin{enumerate}
\item{The equation $g^*(s|_{C_j})=s|_{C_i}$ holds for every birational map 
$g:C_i\dashrightarrow C_j$ such that $\nu|_{C_i}=\nu|_{C_j}\circ g$ holds as rational maps. 
Note that 
$g^*(s|_{C_j})=s|_{C_i}$ means that there exist non-empty open subsets $C_i '\subset C_i$, 
$C_j'\subset C_j$ and 
an isomorphism $g':C_i' \to C_j'$ induced by $g$ such that 
$g'^*(s|_{C_j'})=s|_{C_i'}$.}
\item{For every $i$, there exists $t_i\in H^0(C_i, \nu^*L)$ such that 
$s|_{C_i}=t_i^{\otimes 2}$.}
\end{enumerate}
Then there exists an element $u\in H^0(X, L^{\otimes 2})$ such that $\nu^*u=s$. 
\end{thm}

\begin{proof}
Consider the exact sequence: 
$$0 \to \MO_X \to \nu_*\MO_Y \oplus \MO_D \to \nu_*\MO_C \to 0,$$
which implies 
$$0 \to H^0(X, L^{\otimes 2}) \to H^0(Y, \nu^*L^{\otimes 2}) \oplus H^0(D, L^{\otimes 2}|_D) \to 
H^0(C, \nu^*L^{\otimes 2}|_C).$$
It suffices to show that there exists $t\in H^0(D, L^{\otimes 2}|_D)$ such that $(\nu|_C)^*t=s|_C.$ 
Since $X$ is $S_2$, we can replace $X$ with arbitrary open subscheme $X'$ with ${\rm codim}_X(X\setminus X')\geq 2.$ 
Thus, we may assume that $C$ and $D$ are regular and of pure codimension one. 
We can apply Proposition~\ref{prop-node}. 
Then, by replacing $X$ with its open subscheme, 
$C \to D$ satisfies one of the following properties. 
\begin{enumerate}
\item[(a)]{$C$ is two copies of $D$, that is, $C\simeq D\amalg D$.}
\item[(b)]{$C \to D$ is a finite surjective morphism between integral schemes such that 
$[K(C):K(D)]=2$ and that $K(C)/K(D)$ is separable.}
\item[(c)]{$C \to D$ is a finite surjective morphism between integral schemes such that 
$[K(C):K(D)]=2$ and that $K(C)/K(D)$ is purely inseparable.}
\end{enumerate}
If (a) or (b) holds, then the condition (1) implies that $s|_C$ descends to $D$. 
If (c) holds, then the condition (2) implies that $s|_C$ descends to $D$. 
\end{proof}

\begin{rem}\label{char2-glue}
By the above proof, 
if the characteristic of $k$ is not equal to $2$, then we can drop the second condition (2) in Theorem~\ref{theorem-node}. 
\end{rem}

\section{Abundance theorem for slc surfaces}

The following definition of slc varieties is the same as 
Definition--Lemma 5.10 in \cite{Kollar}. 
For more details, see also \cite[1.41, 5.1, 5.9, 5.10]{Kollar}. 
Moreover, we define sdlt varieties.  

\begin{dfn}\label{def-slc}
Let $X$ be a variety. 
Assume that $X$ is $S_2$ and 
that $X$ is regular or nodal in codimension one. 
Let $\Delta$ be an effective \Q-divisor 
such that $K_X+\Delta$ is \Q-Cartier. 
Let $\nu:Y\to X$ be the normalization and 
we define $\Delta_Y$ by 
$K_Y+\Delta_Y=\nu^*(K_X+\Delta)$. 
We say $(X, \Delta)$ is 
{\em slc} if $(Y, \Delta_Y)$ is lc. 
We say $(X, \Delta)$ is 
{\em sdlt} variety if $(Y, \Delta_Y)$ is dlt and 
every irreducible component of $X$ is normal. 
\end{dfn}

\begin{rem}
(1) 
Note that {\em sdlt} in Definition~\ref{def-slc} and 
{\em semi-dlt} in the sense of \cite[Definition~5.19]{Kollar} are different. 
There is an sdlt variety which is not semi-dlt 
(see the example after \cite[Definition~5.19]{Kollar}). 

(2) 
In characteristic zero, semi-dlt varieties are sdlt by \cite[Definition~5.20]{Kollar}. 
In positive characteristic, we do not know whether the notions of semi-dlt and sdlt have some relations. 
\end{rem}

We recall the $B$-birational maps introduced in \cite{Fujino}. 

\begin{dfn}
Let $(X, \Delta_X)$ and $(Y, \Delta_Y)$ be 
lc varieties (may be reducible). 
We say 
$\sigma:(X, \Delta_X)\dasharrow (Y, \Delta_Y)$ is 
a $B$-{\em birational map} if 
$\sigma:X\dasharrow Y$ is a birational map and 
there exist proper birational morphisms 
$\alpha:W\to X$ and $\beta:W\to Y$ from 
a normal variety $W$ such that 
$\beta=\sigma\circ\alpha$ and 
$\alpha^*(K_X+\Delta_X)=\beta^*(K_Y+\Delta_Y)$. 
Note that $B$-birational maps may permute the irreducible components. 
We define 
$$\Aut(X, \Delta_X):=\{\sigma\in\Aut(X)\,|\,
K_X+\Delta_X=\sigma^*(K_X+\Delta_X)\}.$$
\end{dfn}

To obtain sections on slc varieties, 
we consider the following sections on sdlt varieties. 
The idea is very similar to the admissible sections in \cite{Fujino}.

\begin{dfn}\label{4.1}
Let $(X, \Delta)$ be an $n$-dimensional 
projective sdlt variety with $n\leq 2$. 
Let $X=\bigcup X_i$ be the irreducible decomposition and 
let $\nu:\coprod X_i\to X$ be the normalization. 
We define $\Delta_i$ by 
$K_{X_i}+\Delta_i=(\nu^*(K_X+\Delta))|_{X_i}.$ 
Note that $(X_i, \Delta_i)$ is dlt. 
Let $m$ be a positive integer such that 
$m(K_X+\Delta)$ is Cartier. 
We define $B$-{\em invariant sections} and 
{\em separably gluable sections} as follows. 
\begin{enumerate}
\item{We say $s\in H^0(X, m(K_X+\Delta))$ 
is $B$-{\em invariant} if 
$g^*(s|_{X_j})=s|_{X_i}$ for every 
$B$-birational map 
$g:(X_i, \Delta_{i})\dasharrow (X_j, \Delta_{j})$. }
\item{We say $s\in H^0(X, m(K_X+\Delta))$ 
is {\em separably gluable} if 
$s|_{\coprod_i \llcorner \Delta_i\lrcorner}$ is $B$-invariant. }
\end{enumerate}
We define vector subspaces 
\begin{eqnarray*}
BI(X, m(K_X+\Delta))&:=&\{s\,\,{\rm is\,\,} B-{\rm invariant}\} \subset H^0(X, m(K_X+\Delta))\\
SG(X, m(K_X+\Delta))&:=&\{s\,\,{\rm is\,\,separably\,\,gluable}\} \subset H^0(X, m(K_X+\Delta)).\\
\end{eqnarray*}
Moreover, we define 
$$BI^{(2)}(X, 2m(K_X+\Delta)):=\{t^2\,|\, t\in BI(X, m(K_X+\Delta))\}$$
$$G(X, 2m(K_X+\Delta)):=\{s\,|\, 
s|_{\coprod_i \llcorner \Delta_i\lrcorner} \in BI^{(2)}(\coprod_i \llcorner \Delta_i\lrcorner, 2m(K_X+\Delta)|_{\coprod_i \llcorner \Delta_i\lrcorner})\}$$
We say $s\in H^0(X, 2m(K_X+\Delta))$ is {\em gluable} if $s\in G(X, 2m(K_X+\Delta))$. 
\end{dfn}

\begin{rem}
In characteristic $p\neq 2$, 
we do not need $BI^{(2)}(X, 2m(K_X+\Delta))$ and $G(X, 2m(K_X+\Delta))$. 
For more details, see Remark~\ref{char2-glue} and the proof of Proposition~\ref{4.5}. 
\end{rem}

The following lemma teaches us 
that, in order to obtain sections on an slc surface, 
we should consider gluable sections on a dlt surface. 

\begin{lem}\label{padescendslc}
Let $(X, \Delta)$ be a projective slc surface. 
Let $\nu:Y\to X$ be the normalization and 
let $K_Y+\Delta_Y:=\nu^*(K_X+\Delta)$. 
Let $\mu:(Z, \Delta_Z)\to (Y, \Delta_Y)$ 
be a birational morphism 
from a projective dlt surface $(Z, \Delta_Z)$ 
such that $K_Z+\Delta_Z=\mu^*(K_Y+\Delta_Y)$. 
Then the following assertions hold. 
If $s\in G(Z, 2m(K_Z+\Delta_Z))$, then 
$s=\mu^*\nu^*t$ for some $t\in H^0(X, 2m(K_X+\Delta)).$ 
\end{lem}

\begin{proof}
The assertion holds by Theorem~\ref{theorem-node}. 
\end{proof}

We summarize the basic properties of $B$-invariant sections and (separably) gluable sections.  

\begin{lem}\label{p^eproperties}
Let $(X, \Delta)$ be an $n$-dimensional projective sdlt variety 
with $n\leq 2.$ 
Let $m$ be a positive integer such that 
$m(K_X+\Delta)$ is Cartier. 
\begin{enumerate}
\item{If $s\in BI(X, m(K_X+\Delta))$, 
then $s^2\in BI^{(2)}(X, 2m(K_X+\Delta))$.}
\item{If $t\in BI^{(2)}(X, 2m(K_X+\Delta))$, 
then $t\in BI(X, 2m(K_X+\Delta))$.}
\item{The vector space $BI(X, m(K_X+\Delta))$ 
generates $\mathcal O_X(m(K_X+\Delta))$ 
if and only if 
$BI^{(2)}(X, 2m(K_X+\Delta))$ generates  $\mathcal O_X(2m(K_X+\Delta))$. }
\item{If $s\in SG(X, m(K_X+\Delta))$, 
then $s^2\in G(X, 2m(K_X+\Delta))$. }
\item{If $t\in G(X, 2m(K_X+\Delta))$, 
then $t\in SG(X, 2m(K_X+\Delta))$.}
\item{
If the vector space $SG(X, m(K_X+\Delta))$ 
generates $\mathcal O_X(m(K_X+\Delta))$, then 
$G(X, 2m(K_X+\Delta))$ generates  $\mathcal O_X(2m(K_X+\Delta))$. }
\item{Assume that $X$ is normal and let $S:=\llcorner\Delta\lrcorner\neq 0.$ 
If the map 
$$SG(X, m(K_X+\Delta))\to 
BI(S, m(K_X+\Delta)|_S)$$
is surjective, then so is the map
$$G(X, 2m(K_X+\Delta))\to 
BI^{(2)}(S, 2m(K_X+\Delta)|_S).$$}
\end{enumerate}
\end{lem}

\begin{proof}
(1)(2)(3) These assertions follow from the definition. \\
(4) The assertion follows from 
$(\nu^*s^2)|_{\coprod\llcorner\Delta_i\lrcorner}=((\nu^*s)|_{\coprod\llcorner\Delta_i\lrcorner})^2.$ \\
(5) The assertions follows from (2). \\
(6)(7) The assertions follow from (4). 
\end{proof}

\begin{lem}\label{4.3}
Let $(X, \Delta)$ be a proper lc curve or a 
proper lc surface such that 
$K_X+\Delta$ is semi-ample and 
$S:=\llcorner\Delta\lrcorner\neq 0.$ 
Let $f:=\varphi_{|k(K_X+\Delta)|}:X\to R$ 
be a surjective morphism to 
a projective variety $R$ 
such that $f_*\mathcal O_X=\mathcal O_R.$ 
Let $T:=f(S)$. 
Assume the following conditions. 
\begin{enumerate}
\item[(a)]{$f_*\mathcal O_S=\mathcal O_T$.}
\item[(b)]{There exist sections 
$\{s_i\}_{i=1}^{q}\subset H^0(S, m(K_X+\Delta)|_S)$ 
without common zeros for some $m$. }
\end{enumerate}
Then, for some $r>0$, there exist sections 
$\{u_i\}_{i=1}^{l}\subset H^0(X, rm(K_X+\Delta))$
which satisfy the following conditions. 
\begin{enumerate}
\item{$u_i|_S=s_i^r$ for $1\leq i \leq q$ and 
$u_i|_S=0$ for $q+1\leq i \leq l.$}
\item{$\{u_i\}_{i=1}^{l}$ have no common zeros.}
\end{enumerate}
\end{lem}

\begin{proof}
There is an ample \Q-Cartier \Q-divisor 
$H$ on $R$ such that 
$K_X+\Delta\sim_{\mathbb{Q}}f^*H.$ 
For $r\gg 0$, we have the following commutative diagram. 
$$
\begin{CD}
H^0(X, rm(K_X+\Delta)) @>>> 
H^0(S, rm(K_X+\Delta)|_S)\\
@AA\simeq A @AA\simeq A\\
H^0(R, rmH) @>{\rm surjection}>> 
H^0(T, rmH|_T)\\
\end{CD}
$$
Let $u_1,\cdots,u_q\in H^0(X, rm(K_X+\Delta))$ be lifts of 
$s_1^r,\cdots,s_q^r$ and 
let us consider the following corresponding sections. 
$$
\begin{CD}
u_i @>>> 
s_i^r\\
@AAA @AAA\\
u_i' @>>> 
s_i'\\
\end{CD}
$$
We may assume that $r$ is so large that 
$I_T\otimes \mathcal O_R(rmH)$ is 
generated by global sections where 
$I_T$ is the corresponding ideal to the closed subscheme $T$. 
Let $t_{q+1}',\cdots, t_l'$ be the basis 
of $H^0(R, I_T\otimes \mathcal O_R(rmH))$ 
and let 
$u_{q+1}',\cdots,u_l'$ be its image 
to $H^0(R, rmH)$. 
Then $u_1',\cdots, u_l'$ have no common zeros. 
Thus the corresponding sections 
$u_1,\cdots, u_l$ satisfy the desired properties. 
\end{proof}

The following proposition is the key 
to prove the abundance theorem for slc surfaces.

\begin{prop}\label{4.5}
Let $(X, \Delta)$ be 
a projective dlt surface such that 
$S:=\llcorner\Delta\lrcorner\neq 0$. 
Let $m$ be a sufficiently 
large and divisible integer such that $m\in 2\mathbb Z_{>0}$. 
If $K_X+\Delta$ is nef, then 
the following assertions hold. 
\begin{enumerate}
\item[(a)]{The following map is surjective: 
$$G(X, 2m(K_X+\Delta))\to 
BI^{(2)}(S, 2m(K_X+\Delta)|_S).$$}
\item[(b)]{Assume that 
$BI(S, m(K_X+\Delta)|_S)$ generates 
$\mathcal O_S(m(K_X+\Delta)|_S)$. 
Then 
$G(X, 2m(K_X+\Delta))$ 
generates 
$\mathcal O_X(2m(K_X+\Delta)).$}
\end{enumerate}
\end{prop}

\begin{proof}
We may assume that $X$ is irreducible. 
By the abundance theorem (cf. \cite{Fujita}), 
we obtain 
$f:=\varphi_{|k(K_X+\Delta)|}:X\to R$ 
such that $f_*\mathcal O_X=\mathcal O_R.$ 
Let $f(S)=:T.$ 
Then (1) or (2) holds. 
\begin{enumerate}
\item{$f_*\mathcal O_S= \mathcal O_T.$}
\item{$f_*\mathcal O_S\neq \mathcal O_T.$}
\end{enumerate}

(1)Assume $f_*\mathcal O_S= \mathcal O_T.$ 
By the diagram of the proof of 
Lemma~\ref{4.3}, the map 
$$H^0(X, m(K_X+\Delta))\to 
H^0(S, m(K_X+\Delta)|_S)$$
is surjective. 
Thus the map 
$$SG(X, m(K_X+\Delta))\to BI(S, m(K_X+\Delta)|_S)$$
is also surjective. 
Thus assertion (a) follows from 
Lemma~\ref{p^eproperties}(7). 
We prove (b). 
Since $BI(S, m(K_X+\Delta)|_S)$ 
generates $\mathcal O_S(m(K_X+\Delta)|_S)$, 
$SG(X, m(K_X+\Delta))$ also generates 
$\mathcal O_X(m(K_X+\Delta))$ by Lemma~\ref{4.3}. 
The assertion follows from Lemma~\ref{p^eproperties}(6).

(2)Assume $f_*\mathcal O_S\neq \mathcal O_T.$ 
We can apply Proposition~\ref{2.1} and 
we obtain Proposition~\ref{2.1}(2). 
Then, we have projective morphisms 
$$f:X\overset{g}\to V\to R$$
where $V$ is a smooth projective curve. 

$\underline{\rm Case}$~(2.1s). 
Assume Proposition~\ref{2.1}(2.1s) holds. 
By Lemma~\ref{p^eproperties}(7), 
it is sufficient to prove (a)$'$ and (b)$'$. 
\begin{enumerate}
\item[(a)$'$]{The following map is surjective: 
$$SG(X, m(K_X+\Delta))\to BI(S, m(K_X+\Delta)|_S).$$}
\item[(b)$'$]{Assume that 
$BI(S, m(K_X+\Delta)|_S)$ generates 
$\mathcal O_S(m(K_X+\Delta)|_S)$. 
Then 
$SG(X, m(K_X+\Delta))$ 
generates 
$\mathcal O_X(m(K_X+\Delta)).$}
\end{enumerate}
First we prove (a)$'$. 
Note that 
there is a Galois involution $\iota:S_1\to S_1$ and 
$\iota$ is $B$-birational. 
Let $s\in BI(S, m(K_X+\Delta)|_S)$. 
Since $s$ is $B$-invariant, 
this section $s$ is invariant for $\iota$. 
Thus $s|_{S_1}$ is the pull-back of 
a section $t\in H^0(V, m(D_V)).$ 
Let $u:=g^*t\in H^0(X, m(K_X+\Delta))$. 
We prove that $u|_S=s.$ 
Let $S=\bigcup S_i$ be the irreducible decomposition. 
Since $S$ is reduced, 
we obtain the exact sequence: 
$$0\to \mathcal O_{S}\to 
\bigoplus_{i} \mathcal O_{S_i}.$$
Therefore it is sufficient to prove that 
$u|_{S_i}=s|_{S_i}$ for every $i$. 
For $i=1$, this is clear by the construction. 
Thus we may assume that $S_i$ is $g$-vertical. 
We take a proper birational morphism 
$\lambda:X''\to X$ in Lemma~\ref{dltonepoint}. 
Let $g'':X''\overset{\lambda}\to X\overset{g}\to V.$
Note that $\lambda_*\mathcal O_{S''}=\mathcal O_S$ 
by Lemma~\ref{dltonepoint} 
where $S'':=\llcorner\Delta''\lrcorner$. 
Thus it is sufficient to prove that 
$u''|_{S_{i}''}=s''|_{S_{i}''}$ where 
$u'':=\lambda^*u$, $s'':=\lambda^*s$ and 
$S_{i}''$ is an irreducible component 
of $S''$ such that $g''$-vertical. 
Let $S_1''$ be the proper transform of $S_1$. 
Assume $S_{1}''\cap S_{i}''\neq \emptyset$. 
Note that, since $(X'', \Delta'')$ is dlt, 
the scheme-theoretic intersection 
$S_{1}''\cap S_{i}''$ is reduced. 
Hence, Lemma~\ref{dltonepoint} 
implies $g_{*}''\mathcal O_{S_{i}''}\simeq 
g_{*}''\mathcal O_{S_{1}''\cap S_{i}''}$. 
Since $m(K_{X''}+\Delta'')$ is the pull-back of 
$mD_V$, this means 
$$H^0(S_{i}'', m(K_{X''}+\Delta'')|_{S_{i}''})\simeq 
H^0(S_{1}''\cap S_{i}'', m(K_{X''}+\Delta'')|_{S_{1}''\cap S_{i}''}).$$
By $u''|_{S_{1}''}=s''|_{S_{1}''}$, we have 
$u''|_{S_{1}''\cap S_{i}''}=s''|_{S_{1}''\cap S_{i}''}.$ 
Therefore, by the above isomorphism, 
we see $u''|_{S_{i}''}=s''|_{S_{i}''}.$ 
If $S_{j}''$ satisfies 
$S_{j}''\cap S_{i}''\neq\emptyset$ 
for $S_{1}''\cap S_{i}''\neq\emptyset$, 
then $u''|_{S_{j}''}=s''|_{S_{j}''}$ 
by the same argument as above. 
By the inductive argument, 
if a vertical irreducible component $S_{j}''$ 
is contained in a connected component 
of $S''$ 
which intersects $S_{1}''$, 
then $u''|_{S_{j}''}=s''|_{S_{j}''}.$ 
By Lemma~\ref{2.4} and Proposition~\ref{2.1}, 
every vertical irreducible component 
$S_{i}''$ satisfies this property. 
Therefore, we see 
$u\in SG(X, m(K_X+\Delta))$ such that 
$u|_S=s$. 

Second, we prove (b)$'$. 
We prove that 
$SG(X, m(K_X+\Delta))$ generates 
$\mathcal O_X(m(K_X+\Delta))$. 
Let $s_1,\cdots, s_r\in BI(S, m(K_X+\Delta)|_S)$ 
be a basis and 
let $u_1,\cdots, u_r\in SG(X, m(K_X+\Delta))$ 
be their lifts. 
Let $t_1,\cdots,t_r\in H^0(V, mD_V)$ 
be the corresponding sections. 
Since $BI(S, m(K_X+\Delta)|_S)$ generates 
$\mathcal O_S(m(K_X+\Delta)|_S)$ and 
$S\to V$ is surjective, 
$t_1,\cdots ,t_r$ have no common zeros. 
Thus the corresponding sections 
$u_1,\cdots, u_r$ generates $\mathcal O_X(m(K_X+\Delta))$. 

$\underline{\rm Case}$~(2.2). 
Assume Proposition~\ref{2.1}(2.2) holds. 
It is sufficient to prove the above assertions (a)$'$ and (b)$'$.  

We prove (a)$'$. 
Note that 
there is a $B$-birational morphism 
$\iota:S_2\to S_1$ obtained by 
$S_2\simeq V\simeq S_1$. 
Let $s\in BI(S, m(K_X+\Delta)|_S).$ 
Since $s$ is $B$-invariant, we see 
$\iota^*(s|_{S_1})=s|_{S_2}$. 
Since $S_1\simeq V$, 
$s|_{S_1}$ is the pull-back of 
a section $t\in H^0(V, mD_V).$ 
Let $u:=g^*t\in H^0(X, m(K_X+\Delta))$. 
We would like to prove that $u|_S=s.$ 
It is sufficient to prove that 
$u|_{S_i}=s|_{S_i}$ for every irreducible component 
$S_i$ of $S$. 
By the same argument as (2.1s), 
it is sufficient to prove this equality 
only for $i=1,2$. 
It is clear in the case where $i=1$. 
Since $\iota^*(u|_{S_1})=u|_{S_2}$, 
it is also clear in the case where $i=2.$ 
The assertion (b) holds by the same argument 
as (2.1s). 

$\underline{\rm Case}$~(2.1i). 
Assume Proposition~\ref{2.1}(2.1i) holds. 
We see $p={\rm char}\,k=2$. 

We prove (a). 
Let $s\in BI^{(2)}(S, 2m(K_X+\Delta)|_S)$. 
Then we have 
$s=\tilde s^{2}$ where 
$\tilde s\in BI(S, m(K_X+\Delta)|_S)$. 
Note that $g|_{S_1}:S_1\to V$ is 
the relative Frobenius morphism. 
Thus the absolute Frobenius morphism $F:S_1\to S_1$ 
factors through $V$: 
$$F:S_1\overset{g|_{S_1}}\to V\overset{G}\to S_1.$$ 
Note that $G$ is a non-$k$-linear isomorphism as schemes and that, 
for an invertible sheaf $L$ on $V$,  
$$G^*(g|_{S_1})^*L\simeq G^*(g|_{S_1})^*G^*(G^{-1})^*L\simeq G^*F^*(G^{-1})^*L\simeq L^{\otimes 2}.$$
We show $\mathcal O_V(2mD_V)\simeq 
G^*\mathcal O_{S_1}(m(K_X+\Delta)|_{S_1})$. 
Since $m\in 2\mathbb Z$, we can write $m=2m'$ where $m'\in\mathbb Z$. 
First, we see  
{\Small $$(g|_{S_1})^*\mathcal O_V(2m'D_V)\simeq
\mathcal O_{S_1}(2m'(K_X+\Delta)|_{S_1})\simeq
(g|_{S_1})^*G^*\mathcal O_{S_1}(m'(K_X+\Delta)|_{S_1}).$$ }
Then, for an invertible sheaf 
$$M:=(G^{-1})^*\mathcal O_V(2m'D_V)\otimes 
\mathcal O_{S_1}(-m'(K_X+\Delta)|_{S_1}),$$ 
we obtain $F^*M=(g|_{S_1})^*G^*M \simeq \mathcal O_{S_1}$. 
This implies 
\begin{eqnarray*}
\mathcal O_V(2mD_V)&\simeq &G^*(g|_{S_1})^*\mathcal O_V(2m'D_V)\\
&\simeq&G^*F^*(G^{-1})^*\mathcal O_V(2m'D_V)\\
&\simeq&G^*F^*\mathcal O_{S_1}(m'(K_X+\Delta)|_{S_1})\\
&\simeq&G^*\mathcal O_{S_1}(m(K_X+\Delta)|_{S_1}).
\end{eqnarray*}

Therefore, the section $s$ 
is the pull-back of 
$$t:=G^*\tilde s\in H^0(V, 2mD_V).$$ 
Let $u:=g^*t\in H^0(X, 2m(K_X+\Delta)).$ 
Then, by the same argument as (2.2s), 
we see $u|_S=s.$ 
This means $u\in G(X, 2m(K_X+\Delta)).$ 

We prove (b), that is, we prove that 
$G(X, 2m(K_X+\Delta))$ generates 
$\mathcal O_X(2m(K_X+\Delta))$. 
Let $s_1,\cdots, s_r\in 
BI^{(2)}(S, 2m(K_X+\Delta)|_S)$ 
be a basis and 
let $u_1,\cdots, u_r\in G(X, 2m(K_X+\Delta))$ 
be their lifts. 
Let $t_1,\cdots,t_r\in H^0(V, 2mD_V)$ 
be the corresponding sections. 
Here, $BI^{(2)}(S, 2m(K_X+\Delta)|_S)$ generates 
$\mathcal O_S(m(K_X+\Delta)|_S)$ 
by Lemma~\ref{p^eproperties}(3). 
Thus, since $S\to V$ is surjective, 
$t_1,\cdots ,t_r$ have no common zeros. 
Thus the corresponding sections 
$u_1,\cdots, u_r$ generates $\mathcal O_X(2m(K_X+\Delta))$. 
\end{proof}

In order to construct $B$-invariant sections, 
we consider the following finiteness theorem. 

%
\begin{thm}\label{finitegroup}
Let $(C, \Delta)$ be a projective lc curve and let $m$ be a positive integer 
such that $m(K_C+\Delta)$ is Cartier. 
Then $\rho_m(\Aut (C, \Delta))$ is a finite group where 
$\rho_m$ is a group homomorphism defined by 
\begin{eqnarray*}
\rho_m: \Aut (C, \Delta)&\to& \Aut (H^0(C, m(K_C+\Delta)))\\
\sigma &\mapsto&(s\mapsto \sigma^*s).
\end{eqnarray*} 
\end{thm}

\begin{proof}
We may assume that $C$ is irreducible. 
If the genus $g(C)\geq 2$, then 
$\Aut (C)$ is a finite group. 
Therefore, $\rho_m (\Aut (C, \Delta))$ is a finite group since $\Aut (C, \Delta)\subset \Aut (C)$. 

If $g(C)=1$ and $\Delta\ne 0$, then 
$\Aut (C, \ulcorner \Delta\urcorner)$ is a quasi-projective 
scheme and $H^0(C, T_C\otimes \mathcal O_C(-\ulcorner \Delta\urcorner))=0$. Therefore, 
$\Aut (C, \ulcorner \Delta\urcorner)$ is a finite group. 
Thus, $\rho _m (\Aut (C, \Delta))$ is a finite group because 
$\Aut (C, \Delta)\subset \Aut (C, \ulcorner \Delta\urcorner)$. 

Assume that $g(C)=1$ and $\Delta=0$. Let $0\in C$ be the origin of the elliptic curve 
$C$. 
Then $T_{-\sigma(0)}\circ \sigma\in \Aut (C, [0])$ for 
any $\sigma \in \Aut (C)$, where 
$T_{-\sigma(0)}$ is the translation of $C$ by $-\sigma(0)$. 
Note that $H^0(C, \mathcal O_C(K_C))\simeq k$ is spanned by 
a translation invariant $1$-form on $C$ and 
that $\Aut (C, [0])$ is a finite group. 
Therefore, $\rho _1(\Aut (C))$ is a finite group. 
Since $\rho_m=\rho_1^{\otimes m}$, $\rho _m (\Aut (C))$ is finite for every $m>0$. 

Finally, we assume that $C=\mathbb P^1$. 
If $|\Supp \Delta|\geq 3$, 
then $\Aut (C, \Delta)$ is a finite group. 
If $\deg (K_C+\Delta)<0$, then there is 
nothing to prove. 
Therefore, we can reduce the problem to the case when $\Delta=\llcorner 
\Delta\lrcorner =\{\text{two points}\}$. 
In this case, we can easily check that 
$\rho _m (\Aut (C, \Delta))$ is finite for every $m>0$. Moreover, 
$\rho_m (\Aut (C, \Delta))$ is trivial if $m$ is an even positive integer. 
\end{proof}

The following proposition 
shows that the assumption of (b) 
in Proposition~\ref{4.5} holds. 

\begin{prop}\label{4.7}
Let $(X, \Delta)$ be a projective lc curve. 
If $K_X+\Delta$ is nef, 
then $BI(X, m'(K_X+\Delta))$ generates 
$\mathcal O_X(m'(K_X+\Delta))$ for some integer $m'>0$. 
\end{prop}

\begin{proof}
We see that $H^0(X, m(K_X+\Delta))$ 
generates $\mathcal O_X(m(K_X+\Delta))$ for some integer $m>0$. 
Let $G:=\rho_{m}(\Aut(X, \Delta)).$ 
Note that this group is finite 
by Theorem~\ref{finitegroup}. 
Let $N:=|G|$ and 
let $G=\{g_1,\cdots, g_N\}$. 
For $1\leq i\leq N$, 
let $\sigma_i$ be the $N$-variable elementary symmetric 
polynomial of degree $i$. 
If $s\in H^0(X, m(K_X+\Delta))$, then 
$$(\sigma_i(g_1^*s,\cdots, g_N^*s))^{N!/i}
\in BI(X, N!m(K_X+\Delta)).$$
Since 
$$\bigcap_{j=1}^N\{g_j^*s=0\}=
\bigcap_{i=1}^{N}\{\sigma_i(g_1^*s,\cdots,g_N^*s)=0\},$$
$BI(X, N!m(K_X+\Delta))$ generates 
$\mathcal O_X(N!m(K_X+\Delta))$. 
\end{proof}

Let us prove the main theorem of this paper. 

\begin{thm}\label{maintheorem}
Let $(X, \Delta)$ be a projective slc surface. 
If $K_X+\Delta$ is nef, then 
$K_X+\Delta$ is semi-ample. 
\end{thm}

\begin{proof}
Let $\nu:Y\to X$ be the normalization and 
we define $\Delta_Y$ by $K_Y+\Delta_Y=\nu^*(K_X+\Delta)$. 
There exists a birational morphism 
$\mu:Z\to Y$ 
from a projective dlt surface $(Z, \Delta_Z)$ 
where $K_Z+\Delta_Z=\mu^*(K_Y+\Delta_Y).$ 
By Lemma~\ref{padescendslc}, 
it is sufficient to prove that 
$G(Z, m_0(K_Z+\Delta_Z))$ generates 
$\mathcal O_Z(m_0(K_Z+\Delta_Z))$ for some $m_0>0$. 
This follows from Proposition~\ref{4.5}(b) and Proposition~\ref{4.7}.
\end{proof}

\section{Appendix: Fundamental properties of dlt surfaces}

We summarize fundamental properties for dlt surfaces. 
In this section, we assume that 
all surfaces are irreducible. 
The results in this section may be well-known for experts. 

First, we recall the definition of dlt surfaces. 
It is easy to see that the following definition is equivalent to 
\cite[Definition~2.8]{Kollar} and \cite[Definition~2.37]{KM}. 

\begin{dfn}\label{defdlt}
Let $X$ be a normal surface and 
let $\Delta$ be a \Q-divisor such that 
$K_X+\Delta$ is \Q-Cartier and 
$0\leq \Delta\leq 1$. 
Let 
{\small 
$$S(X, \Delta):=$$
$${\rm Sing}(X)
\cup
\{x\in {\rm Reg}(X)\,|\,
\Supp\,\Delta\,\,{\rm is\,\,not
\,\,simple\,\,normal\,\,crossing
\,\,at}\,\,x\}.$$}
We say $(X, \Delta)$ is {\em dlt} if 
$a(E, X, \Delta)>-1$ 
for every proper birational morphism 
$f:Y\to X$ and 
every $f$-exceptional 
prime divisor $E\subset Y$ such that 
$f(E)\in S(X, \Delta)$.
\end{dfn}

\begin{prop}\label{Szabo}
Let $X$ be a normal surface and 
let $\Delta$ be a \Q-divisor such that 
$K_X+\Delta$ is \Q-Cartier and 
$0\leq \Delta\leq 1.$ 
The following assertions are equivalent: 
\begin{enumerate}
\item{$(X, \Delta)$ is dlt.}
\item{There exists a projective birational 
morphism $\mu:X'\to X$ 
from a smooth surface 
such that 
$\Ex(\mu)\cup \Supp\,\mu^{*}(\Delta)$
is a simple normal crossing divisor and 
each $\mu$-exceptional prime divisor $E_i$
satisfies $a(E_i, X, \Delta)>-1$.}
\end{enumerate}
\end{prop}

\begin{proof}
Note that $S(X, \Delta)$ is a finite set. 

Assume $(1)$, that is, assume that 
$(X, \Delta)$ is dlt. 
Let $f:Y\to X$ be a log resolution of $(X, \Delta)$. 
Let 
$$\Ex(f):=E_1\amalg\cdots\amalg E_r
\amalg F_1\amalg\cdots\amalg F_s$$ 
be the decomposition into the connected components 
where $P_i:=f(E_i)\in S(X,\Delta)$ and 
$Q_j:=f(F_j)\not\in S(X,\Delta)$. 
There exists a proper birational morphisms 
$$Y\overset{g}\to Z\overset{h}\to X$$ 
such that $Z$ is a normal surface and 
$\Ex(g)=F_1\sqcup\cdots\sqcup F_s.$
Indeed, 
$Z$ is obtained by glueing the varieties 
$X\setminus \{P_1,\cdots,P_r\}$ and 
$Y\setminus (F_1\sqcup\cdots\sqcup F_s)$. 
Note that this morphism $h:Z\to X$ is projective 
because $Z$ is smooth. 
Thus this morphism satisfies $(2)$. 

Assume $(2)$. 
Let $f:Y\to X$ be a proper birational morphism 
and let $E\subset Y$ be a prime divisor such that 
$f(E)\in S(X, \Delta)$. 
We prove $a(E, X, \Delta)>-1.$ 
We may assume that 
there exists a proper birational morphism 
$Y\overset{f'}\to X'$ and $Y$ is smooth 
by replacing $Y$ with a desingularization of 
a resolution of indeterminacy 
$Y\dashrightarrow X'$. 
There are two cases: $\dim f'(E)=0$ and 
$\dim f'(E)=1.$ 
The latter case is clear by $(2)$. 
Thus we may assume $f'(E)$ is one point. 
Let $K_{X'}+\Delta':=\mu^*(K_X+\Delta)$. 
Since $f(E)\in S(X, \Delta)$, 
there exists an $\mu$-exceptional curve $E_i$ 
such that $f'(E)\in E_i.$ 
We can write the prime decomposition 
$$\Delta':=b_iE_i+\cdots$$
where $b_i<1.$ 
Then we see that $a(E, X, \Delta)>-1$ 
since $\Delta'$ is simple normal crossing and 
since the morphism $f':Y\to X'$ is a sequence 
of blow-ups. 
\end{proof}

\begin{prop}\label{qfacdlt}
Let $(X, \Delta)$ be a dlt surface. 
Then $X$ is \Q-factorial. 
\end{prop}

\begin{proof}
See, for example, \cite[Theorem~5.3]{Tanaka1}.
\end{proof}

\begin{prop}
Let $(X, \Delta)$ be a dlt surface. 
If a \Q-divisor $\Delta'$ satisfies 
$0\leq \Delta'\leq\Delta$, then 
$(X, \Delta')$ is dlt. 
\end{prop}

\begin{proof}
Since $X$ is \Q-factorial, 
the assertion immediately follows 
from Definition~\ref{defdlt}. 
\end{proof}

\begin{prop}\label{pltdlt}
Let $(X, \Delta)$ be a dlt surface. 
Then the following assertions are equivalent. 
\begin{enumerate}
\item{$(X, \Delta)$ is plt.}
\item{$\llcorner \Delta\lrcorner$ is smooth.}
\item{Each connected component of 
$\llcorner \Delta\lrcorner$ is irreducible.}
\end{enumerate}
\end{prop}

\begin{proof}
See \cite[Proposition~5.51]{KM}. 
Note that the proof of \cite[Proposition~5.51]{KM} 
needs the relative Kawamata--Viehweg vanishing theorem 
for a resolution of singularities $Y\to X$. 
This follows from \cite{Tanaka2}.
\end{proof}

\begin{cor}
Let $(X, \Delta)$ be a dlt surface. 
Then each prime component of 
$\llcorner \Delta\lrcorner$ is smooth.
\end{cor}

\begin{proof}
Let $C$ be a prime component of $\llcorner \Delta\lrcorner$. 
Then $(X, C)$ is plt by Proposition~\ref{pltdlt}. 
\end{proof}

\begin{prop}\label{dltboundary}
Let $(X, C+\Delta')$ be a dlt surface where 
$C$ is a smooth curve in $X$. 
Let $(K_X+C+\Delta')|_C=:K_C+\Delta_C$. 
Then $(C, \Delta_C)$ is lc, that is, 
$0\leq \Delta_C\leq 1$. 
\end{prop}

\begin{proof}
Let $f:Y\to X$ be an arbitrary resolution and 
let $C_Y$ be the proper transform of $C$. 
Let $f^*(K_X+C+\Delta')=:K_Y+C_Y+\Delta_Y$. 
Note that $C\simeq C_Y$. 
Consider the following commutative diagram. 
$$
\begin{CD}
C_Y @>{\rm closed}>{\rm immersion}> Y\\
@V\simeq VV @VVfV\\
C @>{\rm closed}>{\rm immersion}> X\\
\end{CD}
$$

We prove that $\Delta_C$ is effective. 
Let $f$ be the minimal resolution. 
Then $\Delta_Y$ is effective and 
$C_Y$ is not a prime component of $\Delta_Y$. 
Thus we have $0\leq \Delta_C$ 
by the adjunction formula. 

Let $f$ be a log resolution. 
Then, by Definition~\ref{defdlt}, 
we see $\Delta_Y\leq 1$. 
This means $\Delta_C\leq 1.$
\end{proof}
\begin{cor}\label{dltsdlt}
Let $(X, \Delta)$ be a dlt surface. 
Assume $S:=\llcorner\Delta\lrcorner\neq 0$ and 
let $K_S+\Delta_S:=(K_X+\Delta)|_S$. 
Then $S$ is normal crossing and $(S, \Delta_S)$ is sdlt. 
\end{cor}

\begin{proof}
By \cite[Theorem~4.15]{KM}, $S$ is normal crossing. 
Thus, the assertion follows 
from Proposition~\ref{dltboundary}. 
\end{proof}

\begin{prop}\label{dltblowup}
Let $(X, \Delta)$ be an lc surface. 
Then there exists a proper birational morphism 
$h:Z\to X$ from a smooth surface $Z$ 
such that $(Z, \Delta_Z)$ is dlt 
where $\Delta_Z$ is defined by 
$K_Z+\Delta_Z=h^*(K_X+\Delta)$. 
\end{prop}

\begin{proof}
Let $f:Y\to X$ be a log resolution of $(X, \Delta)$ and 
let $K_Y+\Delta_Y:=f^*(K_X+\Delta)$. 
Let $\Delta_Y=\Delta_Y^{+}-\Delta_Y^{-}$ 
where $\Delta_Y^{+}$ and $\Delta_Y^{-}$ 
are effective and 
$\Delta_Y^{+}$ and $\Delta_Y^{-}$ 
have no common irreducible components. 
Since $K_Y\cdot \Delta_Y^{-}<0$ and 
each irreducible component of 
$\Delta_Y^{-}$ is $f$-exceptional, 
there exists a $(-1)$-curve $C$ 
such that $C\subset \Supp \Delta_Y^{-}.$ 
Contract this $(-1)$-curve $Y\to Y'$. 
We repeat this procedure and we obtain 
morphisms 
$$f:Y\overset{g}\to Z\overset{h}\to X.$$
Then we see that 
$Z$ is smooth and $0\leq \Delta_Z\leq 1$ 
where $K_Z+\Delta_Z=h^*(K_X+\Delta).$ 
We prove that $(Z, \Delta_Z)$ is dlt. 
Let $l:W\to Z$ be a proper birational morphism and 
$E\subset W$ be an $l$-exceptional prime divisor such that 
$l(E)\in S(Z, \Delta).$ 
We prove $a(Z, \Delta_Z, E)>-1.$ 
We may assume that 
$W$ is smooth and $l:W\to Z$ factors through $Y$. 
We obtain four surfaces: 
$$W\overset{p}\to Y\overset{g}\to Z\overset{h}\to X.$$
Note that $p(E)\subset \Supp \Delta_Y^{-}.$ 
There are two cases: 
(0)$\dim p(E)=0$ and (1)$\dim p(E)=1$. 

(0)Assume $\dim p(E)=0$. 
Note that $p$ is a composition of blow-ups. 
Since $\Delta_Y$ is simple normal crossing and  $p(E)\in \Supp \Delta_Y^{-}$, 
we obtain $a(Z, \Delta_Z, E)>0$ by a direct calculation. 

(1)Assume $\dim p(E)=1$. 
Since $p(E)\subset \Supp \Delta_Y^{-}$, 
we obtain the inequality $a(Z, \Delta_Z, E)>0.$ 
\end{proof}

\begin{lem}\label{dltonepoint}
Let $(X, \Delta)$ be a dlt surface. 
Then there exists a proper birational morphism 
$\lambda:X''\to X$ from a normal surface $X''$ 
which satisfies the following properties. 
\begin{enumerate}
\item{For $K_{X''}+\Delta'':=\lambda^*(K_X+\Delta)$, 
the pair $(X'', \Delta'')$ is dlt. }
\item{If $S_i''$ and $S_j''$ are prime components 
of $\llcorner\Delta''\lrcorner$ such that 
$S_i''\neq S_j''$ and $S_i''\cap S_j''\neq\emptyset$, 
then $S_i''\cap S_j''$ is one point. }
\item{$\lambda_*(\llcorner\Delta''\lrcorner)=
\llcorner\Delta\lrcorner$ and 
$\lambda_*\mathcal O_{\llcorner\Delta''\lrcorner}=
\mathcal O_{\llcorner\Delta\lrcorner}$.}
\end{enumerate}
\end{lem}

\begin{proof}
If $(X, \Delta)$ satisfies the condition (2), 
then the assertion is clear. 
Thus we may assume that 
there exists prime components $S_i$ and $S_j$ 
of $\llcorner\Delta\lrcorner$ such that 
$S_i\neq S_j$ and 
$S_i\cap S_j$ has at least two points. 
Let $P\in S_i\cap S_j$. 
Note that, since $(X, \Delta)$ is dlt, 
$P\in {\rm Reg}(X)$ and 
$\Supp \Delta$ is simple normal crossing at $P$. 
Let $\mu:Y\to X$ be the blowup at $P$ and 
let $K_Y+\Delta_Y:=\mu^*(K_X+\Delta)$. 
We apply this argument to $(Y, \Delta_Y)$ and 
we repeat the same procedure. 
Then, by a direct calculation and Lemma~\ref{1.3}, 
we obtain the desired morphism $X''\to X.$ 
\end{proof}

\end{document}